\documentclass[11pt]{article}
\usepackage{amsmath,amssymb,amsthm,bm,mathrsfs}
\numberwithin{equation}{section}

\usepackage{relsize,exscale}
\usepackage{graphicx}
\graphicspath{{./figs/}}
\usepackage[lofdepth,lotdepth]{subfig}
\usepackage{float}
\usepackage{enumitem}

\usepackage{xcolor}

\usepackage[pagebackref,bookmarks=false]{hyperref}
\hypersetup{
	plainpages=false,
    colorlinks,
    linktocpage=true,   
    linkcolor={red!50!black},
    citecolor={blue!50!black},
    urlcolor={blue!80!black}
}

\usepackage[margin=1.4in]{geometry}

\title{\texorpdfstring{Robust A Posteriori Error 
Estimation for Finite Element Approximation to 
$\boldsymbol{H}(\mathbf{curl})$ Problem }
{Robust A Posteriori Error 
Estimation for Finite Element Approximation to H(curl) Problem}\thanks{
This work performed under the auspices of the U.S. Department of Energy by 
Lawrence Livermore National Laboratory under Contract DE-AC52-07NA27344 
(LLNL-JRNL-645325). 
This work was supported in part by the National Science Foundation
under grants DMS-1217081, DMS-1320608, DMS-1418934, and DMS-1522707.}}

\author{Zhiqiang Cai\thanks{
Department of Mathematics, Purdue University, 150 N. University
Street, West Lafayette, IN 47907-2067, zcai@math.purdue.edu.}
\and Shuhao Cao\thanks{
Department of Mathematics, Pennsylvania State University, University Park, State 
College, PA 16802, scao@psu.edu.}
\and Rob Falgout\thanks{Center for Applied Scientific Computing, 
Lawrence Livermore National Laboratory, 
Livermore, CA 94551-0808, falgout2@llnl.gov.}}

\date{}

\begin{document}
\maketitle

\newcommand{\vn}{\bm{n}}
\newcommand{\vf}{\bm{f}}
\newcommand{\vj}{\bm{j}}
\newcommand{\vu}{\bm{u}}
\newcommand{\vv}{\bm{v}}
\newcommand{\vw}{\bm{w}}
\newcommand{\vp}{\bm{p}}
\newcommand{\vg}{\bm{g}}
\newcommand{\vr}{\bm{r}}
\newcommand{\vs}{\bm{s}}
\newcommand{\vt}{\bm{t}}
\newcommand{\vx}{\bm{x}}
\newcommand{\ve}{\bm{e}}
\newcommand{\vz}{\bm{z}}

\newcommand{\vH}{\bm{H}}
\newcommand{\vE}{\bm{E}}
\newcommand{\vX}{\bm{X}}
\newcommand{\vL}{\bm{L}}
\newcommand{\vP}{\bm{P}}
\newcommand{\vU}{\bm{U}}
\newcommand{\RR}{\mathbb{R}}
\newcommand{\fI}{\mathfrak{I}}
\newcommand{\sP}{\mathscr{P}}
\newcommand{\fA}{\mathfrak{A}}

\def\scT{{_\cT}}
\def\sD{{_D}}
\def\sF{{_F}}
\def\sK{{_K}}
\def\sI{{_I}}
\def\sb{{_b}}
\def\sN{{_N}}
\newcommand{\cD}{{\mathcal{D}}}
\newcommand{\cF}{{\mathcal{F}}}
\newcommand{\cI}{{\mathcal{I}}}
\newcommand{\cJ}{{\mathcal{J}}}
\newcommand{\cL}{{\mathcal{L}}}
\newcommand{\cN}{{\mathcal{N}}}
\newcommand{\cT}{{\mathcal{T}}}

\newcommand{\W}{{\mathbf W}}
\newcommand{\cA}{{\mathcal A}}
\newcommand{\cB}{\mathcal{B}}
\newcommand{\cE}{\mathcal{E}}
\newcommand{\cM}{\mathcal{M}}
\newcommand{\cP}{\mathcal{P}}
\newcommand{\cR}{\mathcal{R}}
\newcommand{\cS}{\mathcal{S}}
\newcommand{\cW}{\mathcal{W}}
\newcommand{\cH}{\mathcal{H}}

\renewcommand{\a}{\alpha}
\renewcommand{\b}{\beta}
\renewcommand{\d}{\delta}
\newcommand{\D}{\Delta}
\newcommand{\s}{\sigma}
\newcommand{\lam}{\lambda}
\newcommand{\Om}{\Omega}
\newcommand{\om}{\omega}
\def\bvphi{{\bm{\varphi}}}
\newcommand{\btau}{\bm{\tau}}
\newcommand{\bsig}{\bm{\sigma}}
\newcommand{\bxi}{\bm{\xi}}
\newcommand{\btheta}{\bm{\theta}}
\newcommand{\bz}{\bm{\zeta}}
\newcommand{\bgam}{\bm{\gamma}}
\newcommand{\sd}{\bsig^{\Delta}}
\newcommand{\C}{\rm I\kern-.5emC}
\newcommand{\R}{\rm I\kern-.19emR}
\newcommand{\G}{\Gamma}
\newcommand{\g}{\gamma}
\newcommand{\bphit}{\bm{\phi}_{\top}}
\newcommand{\wtauz}{\wt{\btau}^{\D}_{\omz}}
\newcommand{\wtsigz}{\wt{\bsig}^{\D}_{\omz}}
\newcommand{\subcT}{\raisebox{-0.5pt}{$\scriptscriptstyle{\cT}$}}
\newcommand{\bsigtz}{\bsig^{\D}_{\vz,\subcT}}
\newcommand{\omz}{\omega_{\bm{z}}}

\newcommand{\argmin}{\mathrm{argmin}}
\newcommand{\wt}{\widetilde}
\newcommand{\ol}{\overline}

\newcommand{\dint}{\displaystyle\int}
\newcommand{\binprod}[2]{\bigl( {#1},\,{#2} \bigr)}
\newcommand{\sprod}{\mathop{\mathchoice  
{\textstyle\prod}  {\prod} {\prod} {\prod} }\nolimits}
\newcommand{\vsprod}{\mathop{\mathchoice  
{\textstyle\bm{\prod}}  {\bm{\prod}} {\bm{\prod}} {\bm{\prod}} }\nolimits}

\newcommand{\bzero}{{\bf 0}}
\newcommand{\p}{\partial}
\newcommand{\at}[1]{\big\vert_{\raisebox{-0.5pt}{\scriptsize$#1$}} }
\newcommand{\cross}{\!\times\!}
\newcommand{\curlt}{\nabla \cross}
\newcommand{\curl}{\mathbf{curl}\hspace{0.7pt}}
\newcommand{\curll}{\operatorname{curl}}
\renewcommand{\div}{\operatorname{div}}
\newcommand{\divv}{\nabla\!\cdot\!}
\newcommand{\dualp}[2]{{\left\langle {#1},{#2} \right\rangle}}
\newcommand{\ljump}{\lbrack\hspace{-1.5pt}\lbrack}
\newcommand{\rjump}{\rbrack\hspace{-1.5pt}\rbrack}
\newcommand{\jump}[2]{\ljump {#1} 
\rjump_{\raisebox{-2pt}{\scriptsize$#2$}} }
\newcommand{\norm}[1]{\left\|{#1}\right\|}
\makeatletter
\newcommand{\enorm}{\@ifstar\@enorms\@enorm}
\newcommand{\@enorms}[1]{%
\left|\mkern-2.5mu\left|\mkern-2.5mu\left|
#1
\right|\mkern-2.5mu\right|\mkern-2.5mu\right|
}
\newcommand{\@enorm}[2][]{%
\mathopen{#1|\mkern-2.5mu#1|\mkern-2.5mu#1|}
#2
\mathclose{#1|\mkern-2.5mu#1|\mkern-2.5mu#1|}
}
\makeatother

\newcommand{\parallelslant}{\mathop{\mathchoice  
{\textstyle\mathbin{\!/\mkern-5mu/\!}}
{\scriptstyle\mathbin{\!/\mkern-5mu/\!}} 
{\scriptscriptstyle\mathbin{\!/\mkern-5mu/\!}}
{\scriptscriptstyle\mathbin{\!/\mkern-5mu/\!}}}\nolimits}

\newcommand{\Lt}{L^2(\Omega)}
\newcommand{\vLt}{\vL^2(\Omega)}
\newcommand{\vhdiv}{\vH(\div)}
\newcommand{\vHdiv}{\vH(\div;\Omega)}
\newcommand{\vHcrl}{\vH(\curl;\Omega)}
\newcommand{\vhcrl}{\vH(\curl)}
\newcommand{\oversett}[2]{%
\mathop{#2}\limits^{\vbox to -.1ex{\kern -0.5ex\hbox{$\scriptstyle #1$}\vss}}}
\newcommand{\vHcrlbz}[1]{{\oversett{\circ}{{}\vH}}_{#1}(\curl;\Omega)}
\newcommand{\ND}{\bm{\cN\!\cD}}
\newcommand{\vHcrlb}[1]{\vH_{#1}(\curl;\Omega)}
\newcommand{\Hmh}{H^{-1/2}(\Gamma)}
\newcommand{\Hmhb}[1]{H^{-1/2}(#1)}
\newcommand{\vHo}{\vH^1(\Omega)}
\newcommand{\vXbt}[1]{\vX_{\parallelslant}(#1)}
\newcommand{\vXbp}[1]{\vX_{\perp}(#1)}
\newcommand{\vHhbt}[1]{\vH^{1/2}_{\parallelslant}(#1)}
\newcommand{\vHhbp}[1]{\vH^{1/2}_{\perp}(#1)}
\newcommand{\vHmhbt}[1]{\vH^{-1/2}_{\parallelslant}(#1)}
\newcommand{\vHmhbp}[1]{\vH^{-1/2}_{\perp}(#1)}
\newcommand{\vHmhdiv}[1]{\vH^{-1/2}_{\parallelslant}(\div_{#1}, #1)}
\newcommand{\vHmhcurl}[1]{\vH^{-1/2}_{\perp}(\curll_{#1}, #1)}
\newcommand{\vHhb}[1]{\vH^{1/2}(#1)}
\newcommand{\vHdivw}[1]{\vH(\div;\Omega,#1)}
\newcommand{\vHdivwz}[1]{\vH(\div 0\,;\Omega,#1)}
\newcommand{\vPH}[1]{{P\!\vH}^{#1}(\Om,\sP)}

\newcommand{\BDM}{\bm{\cB\!\cD\!\cM}}
\newcommand{\RT}{\bm{\cR\!\cT}}

\newtheorem{theorem}{Theorem}[section]
\numberwithin{theorem}{section}
\newtheorem{assumption}[theorem]{Assumption}

\newtheorem{remark}[theorem]{Remark}
\theoremstyle{definition}
\newtheorem{definition}[theorem]{Definition}
\newtheorem{lemma}[theorem]{Lemma}
\newtheorem{prop}[theorem]{Proposition}

\begin{abstract}
In this paper, we introduce a novel a posteriori error estimator for the 
conforming finite element approximation to the $\vhcrl$ problem with 
inhomogeneous media and with the right-hand side only in $\vL^2$. The estimator 
is of the recovery type. Independent with the current approximation to the primary 
variable (the electric field), an auxiliary variable (the magnetizing field) is 
recovered in parallel by solving a similar $\vhcrl$ problem. An alternate way of 
recovery is presented as well by localizing the error flux. The estimator is 
then defined as the sum of the modified element residual and the residual of the 
constitutive equation defining the auxiliary variable. It is proved that the 
estimator is approximately equal to the true error in the energy norm without the 
quasi-monotonicity assumption. Finally, we present numerical results for several 
$\vhcrl$ interface problems.
\end{abstract}

\section{Introduction}

Let $\Om$ be a bounded and
simply-connected polyhedral domain in $\RR^3$ 
with boundary $\partial\Omega=\bar{\Gamma}_D\cup\bar{\Gamma}_N$ and 
$\Gamma_D\cap \Gamma_N=\emptyset$, and let $\vn=(n_1,n_2, n_3)$ be the outward unit
vector normal to the boundary. Denote by $\vu$ the electric field, we 
consider the following $\vhcrl$ model problem originated from a second order 
hyperbolic equation by eliminating the magnetic field in Maxwell's equations:
\begin{equation}
\label{eq:pb-ef}
\left\{
\begin{array}{rcll}
\curlt (\mu^{-1}  \curlt \vu) + \b\, \vu &= &\vf,  
&\; \text{ in }\, \Om,
\\[2mm]
\vu \cross \vn &=&\vg_{_D},  &\; \text{ on }\, \G_D,
\\[2mm]
(\mu^{-1} \curlt \vu) \cross \vn & = &\vg_{_N}, 
&\; \text{ on }\, \G_N,
\end{array} 
\right.
\end{equation}
where $\curlt$ is the curl operator;
the $\vf$, $\vg_{_D}$, and $\vg_{_N}$ are given vector fields which are assumed to 
be well-defined on $\Om$, $\G_D$, and $\G_N$, respectively; 
the $\mu$ is the magnetic permeability;
and the $\beta$ depends on the electrical conductivity, the dielectric constant, and 
the time step size.
Assume that the coefficients $\mu^{-1}\in \vL^\infty(\Omega)$ and 
$\b \in \vL^\infty(\Omega)$ are bounded below 
\[
 0<\mu_0^{-1}\leq \mu^{-1} (\vx)
 \quad\mbox{and}\quad
0<\b_0\leq \b (\vx)
\]
for almost all $\vx\in \Omega$.

\smallskip

The \emph{a posteriori} error estimation for the conforming finite element 
approximation to the $\vhcrl$ problem 
in~\eqref{eq:pb-ef} has been studied 
recently by several researchers. Several types of \emph{a posteriori} error 
estimators have been introduced and 
analyzed. These include residual-based estimators and the corresponding 
convergence analysis (explicit 
\cite{Hiptmair2000,Zou-1,Zou-2,Chen10,Nicaise03, 
Nicaise07,Schoberl07}, and 
implicit \cite{Harutyunyan08}), equilibrated estimators \cite{Braess06}, 
and recovery-based estimators \cite{Cai15,Nicaise05}.
There are four types of errors in the explicit 
residual-based estimator (see \cite{Hiptmair2000}). 
Two of them are standard, i.e., the element residual, and the interelement face 
jump induced by the discrepancy induced by integration by parts associated with 
the original equation in \eqref{eq:pb-ef}. The other two are also the element 
residual and the interelement face jump, but associated with the divergence of 
the original equation: $\divv (\b \vu) =\divv \vf$, where $\divv$ is the 
divergence operator. These two quantities measure how good the approximation is in 
the kernel space of the curl operator.

\smallskip

Recently, the idea of the robust recovery estimator explored in \cite{Cai09, 
Cai10} for the diffusion interface problem has been extended to the $\vhcrl$ 
interface problem in \cite{Cai15}. Instead of recovering two quantities in the 
continuous polynomial spaces like the extension of the popular Zienkiewicz-Zhu (ZZ) 
error estimator in \cite{Nicaise05},  two quantities related to $\mu^{-1}\curlt \vu$ 
and $\b \vu$ are recovered in the respective $\vhcrl$- and $\vhdiv$-conforming 
finite element spaces. The resulting estimator consists of four terms similar 
to the residual estimator in the pioneering work \cite{Hiptmair2000} 
on this topic by Beck, Hiptmair, Hoppe, and Wohlmuth: two of them 
measure the face jumps of the tangential components and the normal component of 
the numerical approximations to $\mu^{-1}\curlt \vu$ and $\b \vu$, 
respectively, and the other two are element residuals of the recovery type.

\smallskip

All existing a posteriori error estimators for the $\vhcrl$ problem assume that the 
right-hand side $\vf$ is in $\vhdiv$ or divergence free. This assumption does not 
hold in many applications (e.g. the implicit marching scheme mentioned in 
\cite{Hiptmair98}). Moreover, two terms of the estimators are 
associated with the divergence of the original equation. In the proof, these 
two terms come to existence up after performing the integration by parts for 
the irrotational gradient part of the error, which lies in the kernel of the curl 
operator. One of the key technical tools, a Helmholtz decomposition, used in this 
proving mechanism, relies on $\vf$ being in $\vhdiv$, and fails 
if $\vf \notin \vhdiv$. In \cite{Chen10}, the assumption that $\vf\in \vhdiv$ is 
weakened to $\vf$ being in the piecewise $\vhdiv$ space with respect to the 
triangulation, at the same time, the divergence residual and norm jump are modified 
to incorporate this relaxation. 
Another drawback of using Helmholtz decomposition on the 
error is that it introduces the assumption of the coefficients' quasi-monotonicity 
into the proof pipeline. An interpolant with a coefficient independent stability 
bound is impossible to construct in a ``checkerboard'' scenario (see 
\cite{Petzoldt02} for diffusion case, and \cite{Cai15} for $\vhcrl$ case). To gain 
certain robustness for the error estimator in the proof, one has to assume the 
coefficients distribution is quasi-monotone. However, 
in an earlier work of Chen, Xu, and Zou (\cite{Zou-2}), it is shown that numerically 
this quasi-monotonicy assumption is more of an artifact introduced by the proof 
pipeline, at least for the irrotational vector fields. As a result, we conjecture 
that the divergence related terms should not be part of an estimator if it is 
appropriately constructed. In Section \ref{sec:numex}, some numerical justifications 
are presented to show the unnecessity of including the divergence related terms.
\smallskip

The pioneering work in using the dual problems for a posteriori 
error estimation dates back to \cite{Oden89}. In \cite{Oden89}, Oden, Demkowicz, 
Rachowicz, and Westermann studied the a posteriori error estimation through duality 
for the diffusion-reaction problem. The 
finite element approximation to a dual problem is used to estimate the error for the 
original primal problem (diffusion-reaction). The result shares the same form to the 
Prague-Synge identity (\cite{Prager1947}) for diffusion-reaction problem. The 
method presented in this paper may be viewed as an extension of the duality
method in \cite{Oden89} to the $\vhcrl$ interface problem. The auxiliary magnetizing 
field introduced in Section \ref{sec:aux} is the dual variable resembling the flux 
variable in \cite{Oden89}. The connection is illustrated in details in Section 
\ref{sec:dual}.

Later, Repin (\cite{Repin2007}) proposes a functional type 
a posteriori error estimator of $\vhcrl$ problem, which can be viewed as an 
extension of the general approach in \cite{Oden89}. 
Repin et al (\cite{Neittaanmaki2010}) improve the estimate by assuming 
that the data $\vf$ is divergence free and the finite element approximation is in 
$\vhdiv$. In \cite{Repin2007}, the upper bound is 
established through integration by parts by introducing an auxiliary variable in an 
integral identity for $\vhcrl$. An auxiliary 
variable is recovered by globally solving an $\vhcrl$ finite element approximation 
problem and is used in the error estimator. For the global lower bound, the error 
equation is solved globally in an $\vhcrl$ conforming finite element space. Then the 
solution is inserted into the functional as the error estimator of which the 
maximizer corresponds to the solution to the error equation.
\smallskip

The purpose of this paper is to develop a novel a posteriori error estimator for 
the conforming finite element approximation to the $\vhcrl$ problem in 
(\ref{eq:pb-ef-weak}) that overcomes the above drawbacks of the existing estimators, 
e.g. the Helmholtz decomposition proof mechanism, restricted by the assumption that 
$\vf\in \vHdiv$ or divergence free, which brings in the divergence related 
terms. Specifically, the estimator studied in this paper is of the recovery type, 
requires the right-hand side merely having a regularity of $\vL^2$, and has only two 
terms that measure the element residual and the tangential face jump of the original 
equation. Based on the current approximation to the primary variable $\vu$ (the 
electric field), an auxiliary variable $\bsig$ (the magnetizing field) 
is recovered by approximating a similar auxiliary $\vhcrl$ problem. To 
this end, a multigrid smoother is used to approximate this auxiliary problem, which 
is independent of the primary equation and is performed in parallel with the primary 
problem. The cost is the same order of complexity with computing the 
residual-based estimator, which is much less than solving the original $\vhcrl$ 
problem. 

An alternate route is illustrated as well in Section \ref{sec:loc} by approximating 
a localized auxiliary problem.  
While embracing the locality, the parallel nature using the multigrid smoother is 
gone. The recovery through 
approximating localized problem requires the user to 
provide element residual and tangential face jump of the numerical magnetizing 
field based on the finite element solution of the primary equation.
The estimator is then defined as the sum of the modified element residual and 
the residual of the auxiliary constitutive equation. It is proved that the 
estimator is equal to the true error in the energy norm globally. Moreover, in 
contrast to the mechanism of the proof using Helmholtz decomposition mentioned 
previously, the decomposition is avoided by using the joint energy norm. As 
a result, the new estimator's reliability does not rely on the coefficients 
distribution (Theorem \ref{th:rel}). 

Meanwhile, in this paper, the method and analysis 
extend the functional-type error estimator in \cite{Repin2007} to a 
more pragmatic context by including the mixed boundary conditions, and furthermore, 
the auxiliary variable $\bsig$ is approximated by a fast multigrid smoother, or by 
solving a localized $\vhcrl$ problem on vertex patches, to avoid solving a global 
finite element approximation problem.

Lastly, in order to compare the new estimator 
introduced in this paper with existing estimators, we present numerical results for  
$\vhcrl$ intersecting interface problems. When $\vf \notin \vhdiv$, the mesh 
generated by our indicator is much more efficient than those by existing indicators 
(Section \ref{sec:numex}). 

\section{Primal Problem and The Finite Element Approximation}

Denote by $\vLt$ the space of the square integrable vector fields in  $\RR^3$ 
equipped with the
standard $\vL^2$ norm: $\norm{\vv}_{\om}=\sqrt{(\vv,\,\vv)_{\om}}$, 
where $(\vu,\,\vv)_{\om}:=\dint_{\om} \vu\cdot\vv\,d\vx$ denotes the standard 
$\vL^2$ inner product over an open subset $\om\subseteq \Omega$, when $\om = 
\Omega$, the subscript is dropped for $\norm{\vv}:= \norm{\vv}_{\Omega}$ and 
$(\vu,\,\vv)=(\vu,\,\vv)_{\Omega}$. Let 
\[
\vHcrl := \{\vv\in \vLt: \curlt \vv \in \vLt \},
\]
which is a Hilbert space equipped with the norm
\[
\norm{\vv}_{\vhcrl}=\Bigl(\norm{\vv}^2+\norm{\curlt\vv}^2\Bigr)^{1/2}.
\]
Denote its subspaces by 
\[\begin{array}{l}
\vHcrlb{B}
:= \{\vv\in \vHcrl: \vv \cross \vn = \vg_{_B} \text{ on } \G_{B} \} \\[2mm]
\mbox{and} \quad
\vHcrlbz{B}
:= \{\vv\in \vHcrlb{B}: \vg_{_B}=\bzero 
\}
\end{array}
\]
for $B=D$ or $N$.

For any $\vv\in \vHcrlbz{D}$, multiplying the first equation in \eqref{eq:pb-ef} by 
a suitable test function $\vv$ with vanishing tangential part on $\G_D$, integrating 
over the domain $\Omega$, and using integration by parts formula for 
$\vhcrl$-regular vector fields (e.g. see 
\cite{Buffa-Ciarlet}), we have 
\begin{eqnarray*}
(\vf,\,\vv)
&=& \big(\curlt (\mu^{-1}  \curlt \vu),\,\vv\big)+ ( \b\,\vu,\,\vv)
\\[2mm]
&=& (\mu^{-1}\curlt\vu,\,\curlt\vv)+ (\b\,\vu,\,\vv)
- \int_{\Gamma_{N}}\vg_{_N} \cdot \vv\,dS.
\end{eqnarray*}
Then the weak form associated to problem~\eqref{eq:pb-ef} is 
to find $\vu\in  \vH_D(\curl;\Omega)$ such that 
\begin{equation}
\label{eq:pb-ef-weak}
A_{\mu,\beta}(\vu,\vv)  = f_{_N}(\vv),
 \quad \forall\;  \vv \in \vHcrlbz{D},
\end{equation}
where the bilinear and linear forms are given by 
\[
\begin{gathered}
A_{\mu,\beta}(\vu,\vv) = (\mu^{-1} \curlt \vu,\curlt\vv ) + (\b\, \vu,\vv)
\quad\mbox{and}\quad
f_{_N}(\vv)=(\vf,\vv) + \dualp{\vg_{_N}}{\vv}_{\G_N},
\end{gathered}
\]
respectively. Here, 
$\dualp{\vg_{_N}}{\vv}_{\G_N}=\dint_{\G_N} \vg_{_N}\cdot \vv\,dS$ 
denotes the duality pair over $\G_N$. Denote by
\[
\enorm{\vv}_{\mu,\beta} 
=\sqrt{A_{\mu,\beta}(\vv,\vv)}
\]
the ``energy'' norm induced by the bilinear form 
$A_{\mu,\beta}(\cdot,\,\cdot)$.

\begin{theorem}\label{th:pb-ef-weak}
Assume that $\vf\in \vL^2(\Omega)$, $\vg_{_D}\in \vXbt{\G_{D}}$, and 
$\vg_{_N}\in \vHhbp{\G_{N}}$. Then the weak formulation of \eqref{eq:pb-ef} 
has a unique solution $\vu \in \vH_D(\curl;\Omega)$ satisfying the following a 
priori estimate
\begin{equation}
\enorm{\vu}_{\mu,\beta} 
\leq \|{\b^{-1/2}\vf}\| 
+\norm{\vg_{_D}}_{-1/2,\mu,\b,\G_{D}} 
+\norm{\vg_{_N}}_{1/2,\mu,\b,\G_N}.
\end{equation}
\end{theorem}

\begin{proof}
For the notations and proof, see the Appendix \ref{appendix}.
\end{proof}

\smallskip

\subsection{Finite Element Approximation}

For simplicity of the presentation, only the tetrahedral elements are considered.
Let $\cT=\{K\}$ be a finite element partition of the domain
$\Omega$. Denote by $h_K$ the diameter of the element $K$. Assume
that the triangulation $\cT$ is regular and quasi-uniform. 

Let $\vP_k(K)=P_k(K)^3$ where $P_k(K)$ is the space of polynomials of degree less 
than or equal to $k$.
Let $ \wt{P}_{k+1}(K)$ and $\wt{\vP}_{k+1}(K)$ be the spaces of homogeneous 
polynomials of scalar functions and vector fields.
Denote by the first or second kind N\'{e}d\'{e}lec elements 
(e.g. see~\cite{Monk, Nedelec80})
\[
\ND^{k} = \{\vv \in \vHcrl: \vv\at{K}\in \ND^{k,i}(K)\,\;\forall \,\, K\in\cT\}
\subset \vHcrl,
\]
for $i=1,\,2$, respectively, where the local 
N\'{e}d\'{e}lec elements are given by
\[
\begin{aligned}
&\ND^{k,1}(K) = \{\vp + \vs: \vp\in \vP_k(K), 
\vs \in\wt{\vP}_{k+1}(K) \text{ such that } \vs \cdot \vx = 0\}
\\[2mm]
\mbox{and}\quad &\ND^{k,2}(K) = \{\vp + \nabla s: 
\vp\in \ND^{k,1}(K), s\in \wt{P}_{k+2}(K) \}.
\end{aligned}
\]

For simplicity of the presentation, we assume that both boundary data $\vg_{_D}$ and 
$\vg_{_N}$ are piecewise polynomials, and the polynomial extension (see 
\cite{Demkowicz2009}) of the Dirichlet boundary data as the tangential trace is in 
$\ND^k$. Now, the conforming finite element approximation to 
(\ref{eq:pb-ef}) is to find $\vu_{_\cT}\in \ND^k\cap {\vH}_D(\curl;\Omega)$ such 
that 
\begin{equation}
\label{eq:pb-ef-fem}
A_{\mu,\beta}(\vu_{_\cT},\,\vv) =  f_{_N}(\vv), 
\quad \forall\,\,\vv \in \ND^k\cap \vHcrlbz{D}.
\end{equation}
Assume that $\vu$ and $\vu_{_\cT}$ are the solutions of the problems in 
(\ref{eq:pb-ef})
and (\ref{eq:pb-ef-fem}), respectively, and that $\vu \in \vH^{k+1}(\Om)$, 
$\curlt \vu \in \vH^{k+1}(\Om)$ (When the regularity assumption is not met, one can 
construct a curl-preserving mollification, see \cite{Ern15}), 
by the interpolation result from 
\cite{Monk} Chapter 5 and C\'{e}a's lemma,
one has the following a priori error estimation:
\begin{equation}
\label{eq:est-ef}
\enorm{\vu - \vu_{_\cT}}_{\mu,\beta} 
\leq C\, h^{k+1}
\Big( \|\vu\|_{_{\vH^{k+1}(\Om)}} + \|\curlt\vu\|_{_{\vH^{k+1}(\Om)}}\Big),
\end{equation}
where $C$ is a positive constant independent of the mesh size 
$h=\max\limits_{K\in\cT} h_K$. 

\section{Auxiliary Problem of Magnetizing Field}
\label{sec:aux}

\subsection{Recovery of the magnetizing field}
Introducing the magnetizing field
\begin{equation}
\label{eq:pb-mf}
\bsig = \mu^{-1} \curlt \vu,
\end{equation}
then the first equation in (\ref{eq:pb-ef}) becomes
\begin{equation}
\label{eq:pb-mf-aux}
\curlt \bsig + \b\, \vu = \vf, \; \mbox{ in }\, \Om.
\end{equation}
The boundary condition on $\Gamma_N$ may be rewritten as follows
\[
\bsig\cross \vn= \vg_{_N},\; \text{ on }\, \Gamma_N.
\]

For any $\btau\in \vHcrlbz{N}$, multiplying equation (\ref{eq:pb-mf-aux})
by $\b^{-1} \curlt\btau$, integrating over the domain $\Om$, 
and using integration by parts and (\ref{eq:pb-mf}), we have 
\begin{eqnarray*}
(\b^{-1}\vf,\,\curlt\btau)
&=& (\b^{-1}\curlt\bsig,\,\curlt\btau)+ (\vu,\,\curlt\btau)
\\[2mm]
&=& (\b^{-1}\curlt\bsig,\,\curlt\btau)+ (\curlt\vu,\,\btau)
\\
&& \; +\int_{\Gamma_D} (\vu\cross \vn) \cdot \btau\,ds
-\int_{\Gamma_N} \vu \cdot (\btau \cross \vn)\,ds
\\[1mm]
&=& (\b^{-1}\curlt\bsig,\,\curlt\btau)+ 
(\mu\,\bsig,\,\btau)
+ \int_{\Gamma_{D}}\vg_{_D} \cdot \btau\,ds.
\end{eqnarray*}
Hence, the variational formulation for the magnetizing field is to find 
$\bsig \in  \vHcrlb{N}$ such that 
\begin{equation}
\label{eq:pb-mf-weak}
A_{\beta,\mu}(\bsig,\btau) =  f_{_D}(\btau),
\quad\forall\;\btau \in  \vHcrlbz{N},
\end{equation} 
where the bilinear and linear forms are given by 
\[
A_{\beta,\mu}(\bsig,\btau) = (\beta^{-1} \curlt \bsig,\curlt\btau ) 
+ (\mu\,\bsig,\btau) 
\quad\mbox{and}\quad
f_{_D}(\btau)=(\beta^{-1}\vf,\curlt\btau) -  
\dualp{\vg_{_D}}{\btau}_{\G_D},
\]
respectively. The natural boundary condition for the primary problem 
becomes the essential boundary condition for the auxiliary problem, while the 
essential boundary condition for the primary problem is now incorporated into 
the right-hand side and becomes the natural boundary condition. Denote the 
``energy'' norm induced by 
$A_{\beta,\mu}(\cdot,\,\cdot)$ by
\[
\enorm{\btau}_{\beta,\mu} =\sqrt{A_{\beta,\mu}(\btau,\btau)}.
\]

\begin{theorem}\label{th:pb-mf-weak}
Assume that $\vf\in \vL^2(\Omega)$, $\vg_{_D}\in \vHhbp{\G_{D}}$, and 
$\vg_{_N}\in \vXbt{\G_{N}}$. Then problem~{\em \eqref{eq:pb-mf-weak}} has a unique 
solution $\bsig \in \vH_N(\curl;\Omega)$ satisfying the following a priori estimate
\begin{equation}\label{ap-estimate-mf}
\enorm{\bsig}_{\beta,\mu} 
\leq \|\b^{-1/2}\vf\|
+\norm{\vg_{_D}}_{1/2,\mu,\b,\G_{D}} 
+\norm{\vg_{_N}}_{-1/2,\b,\mu,\G_N}.
\end{equation}
\end{theorem}

\begin{proof}
The theorem may be proved in a similar fashion as Theorem~\ref{th:pb-ef-weak}.
\end{proof}

Similarly to that for the essential boundary condition, it is assumed that the 
polynomial extension of the Neumman boundary data as the tangential trace is in 
$\ND^k$ as well. Now, the conforming finite element approximation to 
(\ref{eq:pb-mf-weak}) is to find $\bsig_{_\cT}\in \ND^k\cap  \vHcrlb{N}$ such that 
\begin{equation}
\label{eq:pb-mf-fem}
A_{\beta,\mu}(\bsig_{_\cT},\,\btau) = f_{_D}(\btau), 
\quad\forall \; \btau \in \ND^k\cap \vHcrlbz{N}.
\end{equation}
Assume that $\bsig$ and $\bsig_{_\cT}$ are the solutions of the problems in 
(\ref{eq:pb-mf})
and (\ref{eq:pb-mf-fem}), respectively, and that $\bsig \in \vH^{k+1}(\Om)$, 
$\curlt \bsig \in \vH^{k+1}(\Om)$,  one has the following a priori
error estimation similar to \eqref{eq:est-ef}
\begin{equation}
\enorm{\bsig - \bsig_{_\cT}}_{\mu,\beta} 
\leq C\, h^{k+1}
\Big( \|\bsig\|_{_{\vH^{k+1}(\Om)}} + \|\curlt\bsig\|_{_{\vH^{k+1}(\Om)}}\Big).
\end{equation}
The a priori estimate shows that heuristically, for the auxiliary magnetizing field 
$\bsig$, using the same order $\vhcrl$-conforming finite element approximation 
spaces with the primary variable $\vu$ may be served as the building blocks for the 
a posteriori error estimation.

\subsection{Localization of the recovering procedure}
\label{sec:loc}

The localization of the recovery of $\bsig_{_\cT}$ for this new 
$\vhcrl$ recovery 
shares similar methodology with the one used in the equilibrated flux recovery (see 
\cite{Braess06,Cai12}). 
However, due to the presence of the $\vL^2$-term, exact equilibration is impossible 
due to several discrepancies: $\curlt \bsig_{_\cT} + \b \vu_{_\cT} \neq \vf$ if 
$\bsig_{_\cT}$ and $\vu_{_\cT}$ are in N\'{e}d\'{e}lec spaces of the same order; 
If $\ND^{k+1}$ is used for $\bsig_{_\cT}$ and $\ND^k$ for $\vu_{_\cT}$, the  
inter-element continuity conditions come into the context in that
$\curlt \ND^{k+1} \subset \RT^k$, which has different inter-element continuity 
requirement than $\ND^k$. Due to these two concerns, the local problem is 
approximated using a constraint $\vhcrl$-minimization.

Let $\bsig^{\D}$ be 
the correction from $\mu^{-1} \curlt\vu_{_\cT}$ to the true magnetizing field: 
$\bsig^{\D} := \bsig - \mu^{-1} \curlt\vu_{_\cT}$. Now $\bsig^{\D}$ can be 
decomposed using a partition of unity: let $\lam_{\vz}$ be the linear Lagrange nodal 
basis function associated with a vertex $\vz\in \cN$, which is the collection of all 
the vertices, 
\begin{equation}
\label{eq:localization}
\bsig^{\D} = \sum_{\vz\in \cN_h} \bsig^{\D}_{\vz}, \text{ with } 
\bsig^{\D}_{\vz} := \lam_{\vz} \bsig^{\D}.
\end{equation}

Denote $\ve_{\vz} := \lambda_{\vz} \ve$. Let the vertex patch 
$\omz:=  \cup_{ \{K\in \cT: \;\vz\in \cN_K \} } K$, where $\cN_K$ is the collection 
of vertices of element $K$. Then the following local problem is what the localized 
magnetizing field correction satisfies:
\begin{equation}
\label{eq:pb-loc-pde}
\left\{
\begin{aligned}
\mu \bsig^{\D}_{\vz} - \curlt \ve_{\vz}  &= - \nabla \lam_{\vz}\cross \ve,
\quad & \text{ in } K\subset \omz,
\\
\curlt\bsig^{\D}_{\vz}  + \b \ve_{\vz} 
&= \lam_{\vz} \vr_K   + \nabla \lam_{\vz} \cross 
(\mu^{-1}\curlt \ve),
\quad & \text{ in } K\subset\omz,
\end{aligned}
\right.
\end{equation}
with the following jump condition on each interior face 
$F\in \cF_{\vz}:= \{F\in \cF: F\in \cF_K 
\text{ for }K\subset \omz,\; F\cap \p\omz = \emptyset\}$, 
and boundary face $F\subset \p \omz$:
\begin{equation}
\begin{aligned}
\label{eq:pb-loc-jumpbc}
\left\{
\begin{aligned}
\jump{\bsig^{\D}_{\vz}\cross \vn_F}{F} &= 
-\lam_{\vz}  \vj_{t,F},
\quad &\text{ on } F\in\cF_{\vz},
\\
\bsig^{\D}_{\vz}\cross \vn_F&= \bm{0}, 
\quad &\text{ on } F\subset \p \omz.
\end{aligned}
\right.
\end{aligned}
\end{equation}
The element residual is $\vr_K:= \bigl(\vf - \b \vu_{_\cT} 
-\curlt (\mu^{-1} \curlt \vu_{_\cT})\bigr)\at{K}$, and the tangential jump is
$\vj_{t,F}:= \jump{(\mu^{-1} \curlt\vu_{_\cT}) \cross \vn_F}{F}$.

To find the correction, following piecewise 
polynomial spaces are defined:
\begin{equation}
\label{eq:space-local}
\begin{aligned}
&\ND^k_{-1}(\omz) = \{ \btau\in \vL^2(\omz):\, 
\btau\at{K}\in \ND^k(K),\; 
\forall K\subset \omz \},
\\[3pt]
&\bm{\cW}^k(\cF_{\vz})= 
\{ \btau\in \vL^2(\cF_{\vz}):\, \btau\at{F}\in \RT^k(F),\; 
\forall F\in\cF_{\vz}; \\
& \hspace{2em} \btau\at{F_i} \cdot (\vt_{ij}\cross \vn_{i}) = 
\btau\at{F_j} \cdot (\vt_{ij}\cross \vn_{j}), \forall F_i, F_j\in \cF_{\vz}, 
\p F_i \cap \p F_j = e_{ij}
\},
\\[3pt]
&\bm{\cH}_{\vz} = \{\btau \in \ND^{k}_{-1}(\omz): \;
\jump{\btau\cross\vn_F}{F} =- \ol{\vj}_{F,\vz}
\; \forall\, F \in \cF_{\vz}\},
\\[3pt]
\text{and  }\;
&\bm{\cH}_{\bzero, \vz} =   \{\btau \in \bm{\cH}_{\vz}: \,
\btau\cross\vn_F\at{F} =\bzero, \; \forall\; F\, \subset \omz\}.
\end{aligned}
\end{equation}
Here $\RT^k(F)$ is the planar Raviart-Thomas space on a given face $F$, of which the 
degrees of freedom can be defined using conormal of an edge with respect to the face 
normal $\vn_F$. For example, $\vt_{ij}$ is the unit tangential vector of edge 
$e_{ij}$ joining face $F_i$ and $F_j$, then the conormal vector of $e_{ij}$ with 
respect to face $F_i$ is $\vt_{ij}\cross \vn_{i}$. $\bm{\cW}^k(\cF_{\vz})$ 
can be viewed as the trace space of the broken N\'{e}d\'{e}lec space 
$\ND^k_{-1}(\omz)$. For detail please refer to Section 4 and 5 in \cite{Cockburn05}.

To approximate the local correction for magnetizing field, 
$\lam_{\vz} \vr_K$ and  $\lam_{\vz} \vj_{t,F}$ are projected onto proper 
piecewise polynomial spaces. To this end, let
\begin{equation}
\label{eq:pb-loc-proj}
\ol{\vr}_{K,\vz} :=  \sprod_K\big(\lam_{\vz} \vr_K\big),
\quad\mbox{and}\quad
\ol{\vj}_{F,\vz} :=  \sprod_F\big(\lam_{\vz} \vj_{t,F}\big),
\end{equation}
where $\sprod_K$ is the $\vL^2$ projection onto the space 
$\RT^{k-1}(K)$, and $\sprod_F$ is the $\vL^2$ projection onto the space 
$\RT^k(F)$. 
Dropping the uncomputable terms in 
\eqref{eq:pb-loc-pde}, and using \eqref{eq:pb-loc-jumpbc} as a constraint, the 
following local $\vhcrl$-minimization problem is to be approximated:
\begin{equation}
\label{eq:pb-loc-min}
\min_{\bsigtz \in \bm{\cH}_{\bzero, \vz}} \Biggl\{\norm{\mu^{1/2} \bsigtz - 
\mu^{-1/2}\curlt 
\ve_{\vz}  }^2_{\omz}
+  \norm{\b^{-1/2}(\curlt\bsigtz  + \b \ve_{\vz} 
-\ol{\vr}_{K,\vz} )}^2_{\omz} \Biggr\}.
\end{equation}
The hybridized problem associated with above minimization is obtained by taking 
variation with respect to $\bsigtz$ of the functional 
by the tangential face jump as a Lagrange multiplier:
\begin{equation}
\label{eq:pb-loc-lagrange}
\begin{aligned}
\cJ^*_{\vz}(\bsigtz,\bxi) &:= 
\frac{1}{2}\norm{\mu^{1/2} \bsigtz - \mu^{-1/2}\curlt \ve_{\vz}  }^2_{\omz}
\\
&\quad 
+ \frac{1}{2}\norm{\b^{-1/2}(\curlt\bsigtz  + \b \ve_{\vz} 
-\ol{\vr}_{K,\vz})}^2_{\omz}
\\
&\quad + \sum_{F\in \cF_{\vz}} \binprod{\jump{\bsigtz\cross\vn_F}{F}
+ \ol{\vj}_{F,\vz}}{\bxi}_F.
\end{aligned}
\end{equation}
For any $\btau \in \ND^k_{-1}(\omz)$, using the fact that $\ve_{\vz} \in 
\vH(\curl; \omz)$, and $\ve_{\vz} = \bm{0}$ on $\p\omz$
\begin{equation}
\begin{aligned}
0=&\; \binprod{\mu^{1/2} \bsigtz - \mu^{-1/2}\curlt \ve_{\vz}  
}{\mu^{1/2}\btau}_{\omz}
\\
&+ \binprod{\b^{-1/2}\curlt\bsigtz  + \b^{1/2} \ve_{\vz}
-\b^{-1/2}\ol{\vr}_{K,\vz}}{\b^{-1/2}\curlt\btau }_{\omz}
\\
&+ \sum_{F\in \cF_{\vz}} \binprod{\jump{\btau\cross\vn_F}{F}}{\bxi}_F
\\
=& \; \binprod{\mu \bsigtz}{\btau}_{\omz}
+ \binprod{\b^{-1} \curlt\bsigtz}{\curlt \btau}_{\omz}
\\
& + \sum_{F\in \cF_{\vz}} \binprod{\jump{\btau\cross\vn_F}{F}}{\bxi-\ve_{\vz}}_F
-\binprod{\b^{-1} \ol{\vr}_{K,\vz} }{\curlt\btau}_{\omz}.
\end{aligned}
\end{equation}

As a result, the local approximation problem is:
\begin{equation}
\label{eq:pb-loc-hyb}
\left\{
\begin{aligned}
&\text{Find } (\bsigtz, \btheta_{\vz}) 
\in \ND^k_{-1}(\omz) \cross  \bm{\cW}^k(\cF_{\vz}) \text{ such that:}
\\[0.5em]
&A_{\b,\mu;{\vz}}\big(\bsigtz,{\btau}\big) 
+  
B_{\vz} \big(\btau, \btheta_{\vz}\big)
= \binprod{\b^{-1} \ol{\vr}_{K,\vz} }{\curlt\btau}_{\omz}, 
\quad \forall\, \btau \in \ND^k_{-1}(\omz),
\\[0.5em]
& \qquad \qquad \qquad\quad 
B_{\vz} \big(\bsigtz, \bgam\big)
= -\sum_{F\in\cF_{\vz}} 
\binprod{\ol{\vj}_{F,\vz}}{\bgam}_F,
\quad  \forall\,  \bgam \in \bm{\cW}^k(\cF_{\vz}),
\end{aligned}
\right.
\end{equation}
wherein the local bilinear forms are defined as follows:
\begin{equation}
\label{eq:pb-loc-bilinear}
\begin{aligned}
A_{\b,\mu;\vz} (\bsig,\btau) &:= 
\binprod{\b^{-1} \curlt\bsig}{\curlt \btau}_{\omz}
+ \binprod{\mu \bsig}{\btau}_{\omz},
\\[0.5em]
\text{ and }\; 
 B_{\vz} (\btau,\bgam) &:= \sum_{F\in\cF_{\vz}} 
\binprod{\jump{\btau\cross\vn_F}{F}}{\bgam}_F.
\end{aligned}
\end{equation}

\begin{prop}
Problem \eqref{eq:pb-loc-hyb} has a unique solution.
\end{prop}
\begin{proof}
For a finite dimensional problem, uniqueness implies existence. It suffices to show 
that letting both the right hand sides be zeros results trivial solution. First by 
$\jump{\btau\cross\vn_F}{} \in \bm{\cW}^k(\cF_{\vz})$ for any 
$\btau \in \ND^k_{-1}(\omz)$ (direct 
implication of Proposition 4.3 and Theorem 4.4 in \cite{Cockburn05}), setting 
$\bgam\at{F} = \jump{\bsigtz\cross\vn_F}{F}$ in the second equation of 
\eqref{eq:pb-loc-hyb} immediately implies that $\jump{\bsigtz\cross\vn_F}{F} 
= 0$. As a result, $\bsigtz\in \vH(\curl; \omz)$. Now let $\btau = \bsigtz$ in the 
first equation of \eqref{eq:pb-loc-hyb}, since $A_{\b,\mu;\vz} (\cdot,\cdot)^{1/2}$
induces a norm in $\vH(\curl; \omz)$, $\bsigtz = \bm{0}$. For $\btheta_{\vz}$, it 
suffices to show that $\btheta_{\vz} = \bm{0}$ on each $F$ if 
\[
\sum_{F\in\cF_{\vz}} \binprod{\jump{\btau\cross\vn_F}{F}}{\btheta_{\vz}}_F = 0,\quad 
\,\forall \btau \in \ND^k_{-1}(\omz)\backslash \vH(\curl; \omz).
\]
Using Theorem 4.4 in \cite{Cockburn05}, if $\btheta_{\vz} \in \bm{\cW}^k(\cF_{\vz})$ 
is non-trivial and satisfies above equation, there always exists a 
$\btau_{\btheta}\in \ND^k_{-1}(\omz)$ such that 
$\jump{\btau_{\btheta}\cross\vn_F}{F} = \btheta\at{F}$. As a result, 
$\sum_{F\in\cF_{\vz}} \norm{\btheta_{\vz}}_{F}^2 = 0$, which is a contradiction. 
Thus, the local problem \eqref{eq:pb-loc-hyb} is uniquely solvable.
\end{proof}

With the local correction to the magnetizing field, $\bsigtz$ for all 
$\vz\in\cN$, computed above, let
\begin{equation}
\wt{\bsig}^{\D}_{K,\cT}=\sum_{\vz\in\cN(K)} \bsigtz,
\quad\text{and}\quad
\wt{\bsig}^{\D}_{\cT} = \sum_{\vz\in\cN} \bsigtz,
\end{equation}
then the recovered magnetizing field is 
\begin{equation}
\label{eq:mf-correction}
\wt{\bsig}_{\cT} = \wt{\bsig}^{\D}_{\cT} + \mu^{-1} \curlt \vu_{_\cT}.
\end{equation}

\section{A Posteriori Error Estimator}

In this section, we study the following a posteriori error estimator:
\[
\eta = \left(\sum_{K\in\cT} \eta_K^2\right)^{1/2}
\]
where the local indicator $\eta_K$ is defined by
\begin{equation}
\label{eq:etaK}
\eta_K 
= 
\left(\norm{\mu^{-1/2}\left(\mu\, \bsig_{_\cT} 
\!- \curlt\vu_{_\cT}\right)}_{K} ^2
+\norm{\b^{-1/2}\left(\curlt \bsig_{_\cT}
 + \b\,\vu_{_\cT}\! - \vf\right) }_{K}^2\right)^{1/2}.
\end{equation}
It is easy to see that 
\begin{equation}
\label{eq:eta}
\eta = 
\left(\norm{\mu^{-1/2}\left(\mu\, \bsig_{_\cT} 
\!- \curlt\vu_{_\cT}\right)}^2
+\norm{\b^{-1/2}\left(\curlt \bsig_{_\cT} 
+ \b\,\vu_{_\cT}\! - \vf\right) }^2\right)^{1/2}.
\end{equation}
The $\vu_{_\cT}$ and $\bsig_{_\cT}$ are the finite element approximations in 
problems~\eqref{eq:pb-ef-fem} and \eqref{eq:pb-mf-fem} respectively.

With the locally recovered $\wt{\bsig}_{\cT}$, the local error indicator 
$\wt{\eta}_K$ and the global error estimator $\wt{\eta}$ 
are defined in the same way as \eqref{eq:etaK} and \eqref{eq:eta}:
\begin{equation}
\label{eq:etaKt}
\wt{\eta}_K 
= 
\left(\norm{\mu^{-1/2}\left(\mu\, \wt{\bsig}_{\cT}
\!- \curlt\vu_{_\cT}\right)}_{K} ^2
+\norm{\b^{-1/2}\left(\curlt \wt{\bsig}_{_\cT}
 + \b\,\vu_{_\cT}\! - \vf\right) }_{K}^2\right)^{1/2},
\end{equation}
and
\begin{equation}
\label{eq:etat}
\wt{\eta} = 
\left(\norm{\mu^{-1/2}\left(\mu\, \wt{\bsig}_{_\cT} 
\!- \curlt\vu_{_\cT}\right)}^2
+\norm{\b^{-1/2}\left(\curlt \wt{\bsig}_{_\cT} 
+ \b\,\vu_{_\cT}\! - \vf\right) }^2\right)^{1/2}.
\end{equation}

\begin{remark}
In practice, $\bsig_{_\cT}$ does not have to be the finite element solution of a 
global problem. In the numerical computation, the Hiptmair-Xu multigrid 
preconditioner in \cite{Hiptmair07} is used for discrete problem 
\eqref{eq:pb-mf-fem} with two $V(1,1)$ 
multigrid V-cycles for each component of the vector Laplacian, 
and two $V(2,2)$ multigrid V-cycles for the kernel part of the curl operator. 
The $\bsig_{_\cT}$ used to evaluate the estimator is the PCG iterate. 
The computational cost is the same order with computing the explicit residual based 
estimator in \cite{Hiptmair2000}.

Generally speaking, to approximate the auxiliary problem, the same black-box
solver for the original $\vhcrl$ problem can be applied requiring minimum 
modifications. For example, if the BoomerAMG in \emph{hypre} 
(\cite{Falgout02,Kolev06}) is used for the discretizations of the primary problem,  
then the user has to provide exactly the same 
discrete gradient matrix and vertex coordinates of the mesh, and in constructing the 
the HX preconditioner, the assembling routines for the vector Laplacian and scalar 
Laplacian matrices can be called twice with only the coefficients input switched.
\end{remark}

\begin{theorem}
\label{th:rel}
Locally, the indicator $\eta_{_K}$ and $\wt{\eta}_{_K}$ both have the following 
efficiency bound
\begin{equation}
\label{eq:efficiency-local}
\eta_{_K}  (\text{or } \wt{\eta}_{_K})
\leq  \,\left( \enorm{\vu-\vu_{_\cT}}_{\mu,\b,K}^2
+\enorm{\bsig-\bsig_{_\cT}}_{\beta, \mu, K}^2\right)^{1/2}
\end{equation}
for all $K\in\cT$. The estimator $\eta$ and $\wt{\eta}$ satisfy the following 
global upper bound
\begin{equation}
\label{eq:rel-global}
\left( \enorm{\vu-\vu_{_\cT}}_{\mu,\b}^2
+\enorm{\bsig-\bsig_{_\cT}}_{\beta,\mu}^2\right)^{1/2} = \eta \leq \wt{\eta} .
\end{equation}
\end{theorem}

\begin{proof}
Denote the true errors in the electric and magnetizing fields by
\[
\ve=\vu-\vu_{_\cT}, \quad\mbox{and}\quad
\vE=\bsig-\bsig_{_\cT},
\]
respectively. It follows from (\ref{eq:pb-mf}), (\ref{eq:pb-mf-aux}), and the 
triangle inequality that 
\begin{eqnarray} \label{eq:identity-local}
\eta_K^2
&=&  \norm{\mu^{1/2} \vE - \mu^{-1/2}\curlt\ve}_{K} ^2
   +\norm{\b^{-1/2}\curlt \vE+ \b^{1/2}\ve}_{K}^2 
   \\[2mm]\nonumber
   &\leq &  \left(\norm{\mu^{1/2} \vE}_{K} ^2 + 
   \norm{\mu^{-1/2}\curlt\ve}_{K} ^2
   + \norm{\b^{-1/2}\curlt \vE}^2_K +\norm{\b^{1/2}\ve}_{K}^2\right)
   \\[2mm]\nonumber
&=& \left(\enorm{\ve}_{\mu,\b,K}^2   +\enorm{\vE}_{\beta,\mu, K}^2\right),
\end{eqnarray}
which implies the validity of \eqref{eq:efficiency-local} for $\eta_{_K}$. For 
$\wt{\eta}_{_K}$, the exact same argument follows except by switching $\vE = \bsig - 
\bsig_{_\cT}$ by locally recovered $\wt{\vE} = \bsig - \wt{\bsig}_{_\cT}$.
To prove the global identity in 
\eqref{eq:rel-global}, summing \eqref{eq:identity-local} over all 
$K\in\cT$ gives
\begin{eqnarray*}
\label{eq:rel-der}
\eta^2
  &=&\norm{\mu^{1/2} \vE - \mu^{-1/2}\curlt\ve} ^2
  +\norm{\b^{-1/2}\curlt \vE+ \b^{1/2}\ve}^2\\[2mm]
  &=& \enorm{\ve}_{\mu,\b}^2
 +\enorm{\vE}_{\beta,\mu}^2
 -2(\vE,\,\curlt\ve)+ 2(\curlt \vE,\,\ve).
\end{eqnarray*}
Now, \eqref{eq:rel-global} follows from the fact that 
\[
-(\vE,\,\curlt\ve)+ (\curlt \vE,\,\ve)=0.
\]
Lastly, the global upper bound for the locally recovered $\wt{\eta}$ follows from 
the fact that $\vu_{_\cT}$ and $\bsig_{_\cT}$ are the solutions to the following 
global problem:
\begin{equation}
\label{eq:pb-mf-min}
\inf_{\substack{\btau \in \vHcrlb{N} \cap \ND^{k}\\
 \vv\in \vHcrlb{D} \cap \ND^{k}}}
\norm{\mu^{-1/2} \left(\mu\,\btau- \curlt\vv\right)}^2
+\norm{\b^{-1/2}\left(\curlt \btau + \b\vv - \vf\right)}^2.
\end{equation}
As a result, $\wt{\eta} \geq \eta$ which is the global minimum achieved in the 
finite element spaces. This completes the proof of the theorem. 
\end{proof}

\begin{remark}
In Theorem \ref{th:rel} it is assumed that the boundary data are admissible so that 
they can be represented as tangential traces of the finite element space $\ND^k$. If 
this assumption is not met, it can be still assumed that divergence-free extension 
$\vu_{\vg_D}$ of its tangential trace $\vg_{_D}$ to each boundary tetrahedron $K$ on 
$\G_D$ is at least $\vH^{1/2+\d}$-regular ($\d>0$), and $\curlt \vu_{\vg_D}\in 
\vL^p(K)$ as well ($p>2$), so that the conventional edge interpolant is well-defined 
(e.g. see \cite{Monk} Chapter 5). When the same assumption is applied to $\bsig$ and 
$\curlt\bsig$, the reliability bound derived by \eqref{eq:rel-der} still holds (for 
notations please refer to Appendix \ref{appendix}):
\[
\eta^2 = \enorm{\ve}_{\mu,\b}^2
 +\enorm{\vE}_{\beta,\mu}^2
+\dualp{\ve\cross \vn}{\pi_{\top, D}\vE}_{\G_D} + 
\dualp{\pi_{\top, N}\ve}{\vE\cross \vn}_{\G_N}.
\]
Using the fact that $\vu$ and $\bsig$ are approximated by the conventional edge 
interpolants $\cI \vu$ on $\G_D$ and $\cI \bsig$ on $\G_N$ respectively yields:
\[
\begin{aligned}
&\dualp{\ve\cross \vn}{\pi_{\top, D}\vE}_{\G_D} 
+ \dualp{\pi_{\top, N}\ve}{\vE\cross \vn}_{\G_N}
\\
\leq & 
\norm{\g_{\top,D}(\vu - \cI \vu)}_{1/2,\mu,\b,\G_{D}}
\norm{\pi_{\top,D} \vE}_{-1/2,\b,\mu,\G_D} 
\\
&+
\norm{\pi_{\top,N} \ve}_{-1/2,\mu,\b,\G_N} 
\norm{\g_{\top,N}(\bsig - \cI \bsig)}_{1/2,\b,\mu,\G_{N}}.
\end{aligned}
\]
By the interpolation estimates for boundary elements together with the weighted 
trace inequalities \eqref{eq:tr} from Appendix \ref{appendix}, the reliability 
constant is not harmed if the interface singularity does not touch the 
boundary.
\end{remark}

\subsection{Relation to Duality Method}
\label{sec:dual}
A posteriori error estimation by the duality method for the diffusion-reaction 
problem was studied by Oden, Demkowicz, Rachowicz, and Westermann in \cite{Oden89}.
In this section, we describe the duality method for the $\vhcrl$ problem and its 
relation with the estimator $\eta$ defined in (\ref{eq:eta}). 

To this end, define the energy and complimentary functionals by
\begin{eqnarray*}
\cJ(\vv) &=&  \frac{1}{2} A_{\mu,\beta}(\vv,\vv) - f_{_N}(\vv)  \\[2mm]
 \mbox{and}\quad \cJ^*(\btau) &=&  -\frac{1}{2}(\mu\,\btau,\,\btau) 
   - \frac{1}{2} \binprod{\b^{-1}(\vf - \curlt\btau)}{\vf-\curlt\btau} 
   -\dualp{\vg_{_D}}{\btau}_{\G_D},
\end{eqnarray*}
respectively. Then problems (\ref{eq:pb-ef-weak}) and (\ref{eq:pb-mf-weak}) are 
equivalent to the following minimization and
maximization problems:
 \[
  \cJ(\vu) = \inf_{\vv \in \vHcrlb{D}} \cJ(\vv)
   \quad\mbox{and}\quad
   \cJ^*(\bsig) = \sup_{\btau \in \vHcrlb{N}} \cJ^*(\btau),
   \]
respectively. By the duality theory for a lower semi-continuous convex functional 
(see e.g. \cite{Ekeland-Temam}), we have
\[
 \cJ(\vu) = \cJ^*(\bsig) 
 \quad\mbox{and}\quad
 \bsig = \mu^{-1}\nabla\times\vu.
\]
   
A simple calculation gives that the true errors of the finite element approximations 
in the ``energy'' norm can be represented by the difference between the functional 
values as follows:
\begin{equation}\label{eq:err-fun}
\enorm{\vu-\vu_{_\cT}}_{\mu,\b}^2 = 2\Big( \cJ(\vu_{_\cT}) - \cJ(\vu)\Big)\,
\mbox{ and }\,
  \enorm{\bsig-\bsig_{_\cT}}_{\beta,\mu}^2 = 2\Big( \cJ^*(\bsig) - 
  \cJ^*(\bsig_{_\cT}) \Big).
\end{equation}
Hence, the ``energy'' error in the finite element approximation is bounded above by 
the estimator $\eta$ defined in (\ref{eq:eta}) (and the locally-recovered 
$\wt{\eta}$ as well):
\begin{eqnarray*}
 \enorm{\vu-\vu_{_\cT}}_{\mu,\b}^2 
 &=& 2\Big(\cJ(\vu_{_\cT}) - \cJ(\vu) \Big)
  = 2 \Big(\cJ(\vu_{_\cT}) - \cJ^*(\bsig) \Big)\\[2mm]
 &\leq & 2\Big( \cJ(\vu_{_\cT}) - \cJ^*(\bsig_{_\cT}) \Big)= \eta^2,
\end{eqnarray*}
where the last equality is obtained by evaluating $ 2\Big(\cJ(\vu_{_\cT}) - 
\cJ^*(\bsig_{_\cT})\Big)$ 
through integration by parts. Note that the above calculation indicates
\[
\eta^2  = \enorm{\vu-\vu_{_\cT}}_{\mu,\b}^2 + 
\enorm{\bsig-\bsig_{_\cT}}_{\beta,\mu}^2 = 
2\Big( \cJ(\vu_{_\cT})  - \cJ^*(\bsig_{_\cT}) \Big),
\]
which leads us back to the identity on the global reliability in
 \eqref{eq:rel-global}.

\section{Numerical Examples}
\label{sec:numex}
In this section, we present numerical results for interface problems, i.e., the 
problem parameters
$\mu$ and $\beta$ in (\ref{eq:pb-ef}) are piecewise constants with respect to a 
partition of the domain $\ol{\Omega}=\cup^n_{i=1} \ol{\Omega}_i$.
Assume that interfaces $\fI=\{\p\Om_i\cap\p\Om_j : i,j=1,\,...,\,n\}$ do not cut
through any element $K\in\cT$. The $\vu_{_\cT}$ is solved in $\ND^0$, and the 
$\bsig_{_\cT}$ is recovered in $\ND^0$ as well.

The numerical experiments are prepared using \texttt{delaunayTriangulation} in MATLAB
for generating meshes, L. Chen's iFEM (\cite{Chen.L2008c}) 
for the adaptively refining procedure, the \texttt{matlab2hypre} interface in BLOPEX 
(\cite{Knyazev07}) for converting sparse matrices, and MFEM (\cite{mfem-library}) to 
set up the serial version of Auxiliary-space Maxwell Solver (AMS) in \emph{hypre} 
(\cite{Falgout02}) as preconditioners. We compare numerical results generated by
adaptive finite element method using following error estimators:
 \begin{itemize}
 \item[(i)] the new indicator $\eta_{_{\text{New},K}}$ defined in~\eqref{eq:etaK}, 
 and its locally-recovered sibling $\wt{\eta}_{_{\text{New},K}}$ defined 
 in~\eqref{eq:etaKt}.
 \item[(ii)] the residual-based indicator $\eta_{_{\text{Res},K}}$ introduced in
 \cite{Hiptmair2000} with the appropriate weights for piecewise constant 
 coefficients defined in  \cite{Cai15}:
  \[
  \begin{aligned}
  \eta_{_{\text{Res},K}}^2
&= \mu_K h_K^2\norm{\vf - \b \vu_{_\cT} - \curlt(\mu^{-1} \curlt 
\vu_{_\cT})}_{\vL^2(K)}^2 + 
\beta_K^{-1} h_K^2\norm{\divv(\b \vu_{_\cT}- \vf)}_{\vL^2(K)}^2
\\[2mm]
&  + \sum_{F\in \cF_h(K)} \frac{h_F}{2} \left( \b_F^{-1} 
\norm{\jump{\b\vu_{_\cT}\cdot \vn_F}{F}}_{L^2(F)}^2 + 
\mu_F\norm{\jump{(\mu^{-1} 
\curlt \vu_{_\cT})\cross \vn}{F}}_{\vL^2(F)}^2 \right),
\end{aligned}
  \]
\item[(iii)] the recovery-based indicator $\eta_{_{\text{Rec},K}}$ presented in 
 \cite{Cai15}:
  \[
   \begin{aligned}
  \eta_{_{\text{Rec},K}}
  &= \mu_K h_K^2\norm{\vf - \b \vu_{_\cT} - \curlt(\mu^{-1} \curlt  
  \vu_{_\cT})}_{\vL^2(K)}^2 
  +  \norm{\b^{-1/2}\btau_{_\cT}-\b^{1/2} \vu_{_\cT}  }_{\vL^2(K)}^2
\\[2mm]
&  \qquad +  \norm{\mu^{1/2}\bsig_{_\cT}-\mu^{-1/2} \curlt  \vu_{_\cT}  
}_{\vL^2(K)}^2,
\end{aligned}
  \]
where $\bsig_{_\cT} \in \ND_0$ and $\btau_{_\cT} \in \BDM_1$ are the $L^2$ 
recoveries of  $\mu^{-1}\curlt  \vu_{_\cT}$ and $\beta \vu_{_\cT}$, respectively.
 \end{itemize}
 
In our computation, the energy norms 
 \[
 \enorm{\vv}_{\mu,\b}
 \quad\mbox{and}\quad
 \enorm{(\vv,\,\btau)}=\left(\enorm{\vv}_{\mu,\b}^2+\enorm{\btau}_{\beta,\mu}^2\right)^{1/2}
 \]
are used for the estimators $\eta_{_\text{Res}}$ and $\eta_{_\text{Rec}}$
and the estimator $\eta_{_\text{New}}$, respectively.
The respective relative errors and effectivity indices are computed at each 
iteration by
\[
\text{rel-error} := 
\frac{\enorm{\vu - \vu_{_{\cT,\#\text{Iter}}}}_{\mu,\b}}
{\enorm{\vu}_{\mu,\b}}
\quad\mbox{and}\quad
\text{eff-index} := \frac{\eta_{_{\# \text{Iter}}}}
{\enorm{\vu - \vu_{_{\cT,\#\text{Iter}}}}_{\mu,\b}}
\]
for the estimators $\eta_{_\text{Res}}$ and $\eta_{_\text{Rec}}$ and by
\[
\text{rel-error} := \frac{\xi}
{\enorm{(\vu,\,\bsig)}} \quad
\mbox{and}\quad 
\text{eff-index} := \frac{\eta_{_{\# \text{Iter}}}}
{\xi}
\]
for the estimator $\eta_{_\text{New}}$ and $\wt{\eta}_{_\text{New}}$, where 
$\xi = \enorm{(\vu-\vu_{_{\cT,\#\text{Iter}}},\, 
\bsig-\bsig_{_{\cT,\#\text{Iter}}})}$.
In all the experiements, the lowest order N\'{e}d\'{e}lec element space is used, 
and, hence, 
the optimal rate of convergence for the adaptive algorithm is 
$O(\texttt{\#DoF}^{-1/3})$.

\smallskip

\textbf{Example 1}: This is an example similar to that in 
\cite{Cai15,Zou-2} with a few additions and tweaks, in which
the Kellogg intersecting interface problem is adapted to the $\vhcrl$-problem. 
The computational domain is a slice 
along $z$-direction: $\Om =(-1, 1)^2\cross (-\d,\d) $ with $\d = 0.25$. 
Let $\a$ be a piecewise constant given by
\[
\a= 
\begin{cases} 
R & \text{ in } (0, 1)^2\cross(-\d,\d)  \cup (−1, 0)^2\cross (-\d,\d) ,
\\
1 & \text{ in } \Om\backslash \Big((0, 1)^2\cross (-\d,\d)  
\cup (−1, 0)^2\cross (-\d,\d) \Big).
\end{cases}
\]
The 
exact solution $\vu$ of (\ref{eq:pb-ef}) is given in cylindrical coordinates 
$(r,\theta,z)$:
\[
\vu =\nabla \psi =  \nabla \bigl(r^{\g} \phi(\theta)\bigr) \in 
\vH^{\gamma-\epsilon}(\Omega)
\,\,\mbox{ for any }\,\, \epsilon >0,
\]
where $\phi(\theta) $ is a continuous function defined by
\[
\phi(\theta) = 
\begin{cases}
\cos\big((\pi/2 - \s)\g\big)\cdot \cos\big((\theta -\pi/2 + \rho)\g\big),
&\text{for } 0 \leq  \theta \leq \pi/2,
\\
\cos(\rho\g) \cdot \cos\big((\theta - \pi + \s)\g\big), 
&\text{for }\pi/2 \leq  \theta \leq \pi,
\\
\cos(\s\g) \cdot \cos\big((\theta - \pi - \rho)\g\big), 
&\text{for }\pi \leq \theta\leq 3\pi/2,
\\
\cos\big((\pi/2-\rho)\g\big) \cdot \cos\big((\theta - 3\pi/2 - \s)\g\big),
&\text{for }3\pi/2 \leq \theta\leq 2\pi.
\end{cases}
\]
Here we set parameters 
to be
\[
 \g = 0.5, \quad R \approx 5.8284271247461907, \quad
 \rho = \pi/4, \,\text{ and }\, \s \approx −2.3561944901923448.
\]
The initial mesh is depicted in Figure 
\ref{fig:ex1-init} which is aligned with four interfaces.
\begin{figure}[h]
\centering
\includegraphics[width=0.35\textwidth]{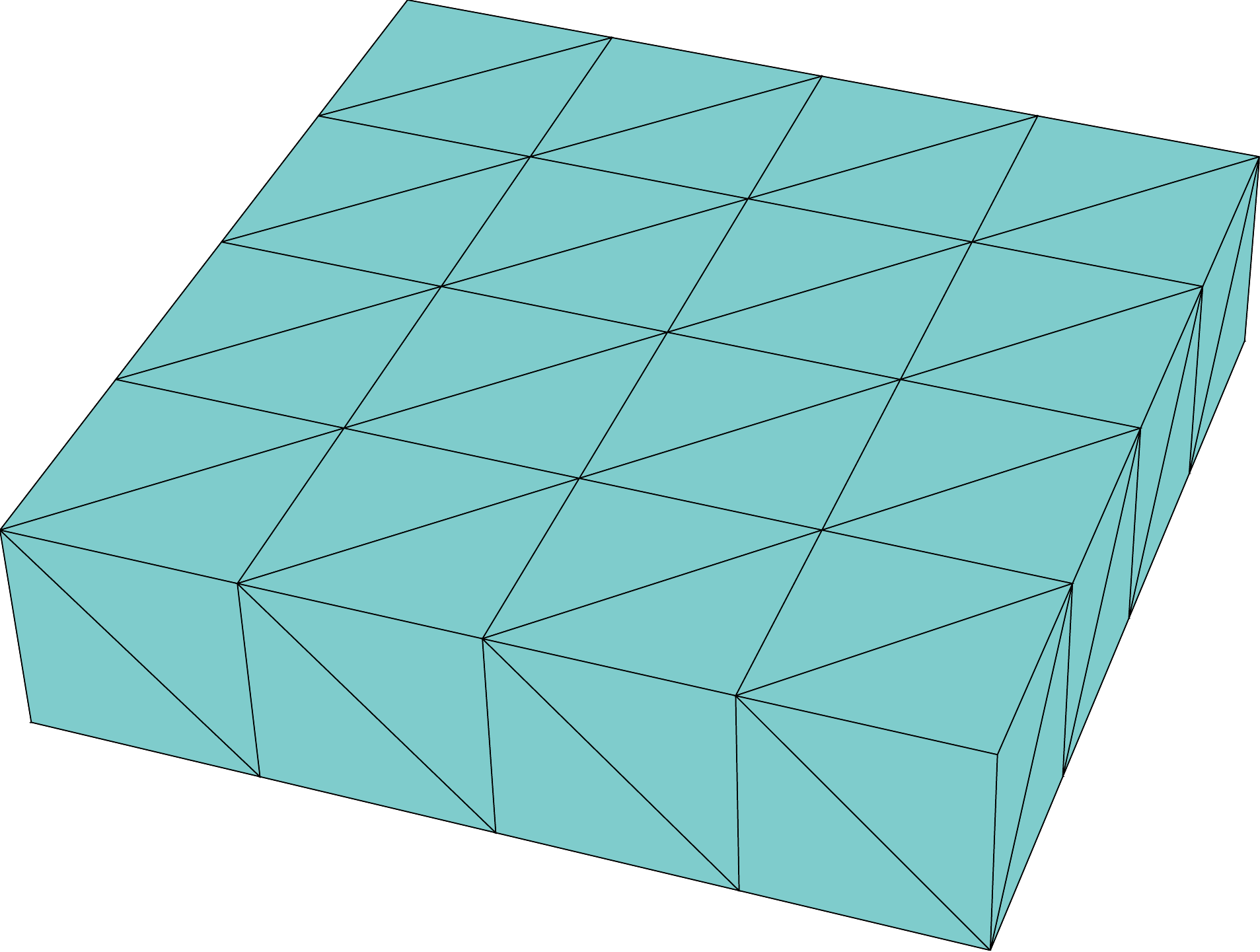}
\caption{Initial Mesh of Example 1}
\label{fig:ex1-init}
\end{figure}

\begin{figure}[h]
\centering
\subfloat[Refined mesh based on $\eta_{_{\text{Res},K}}$ cut on $z\!=\!0$]
{
\includegraphics[width=0.3\textwidth]{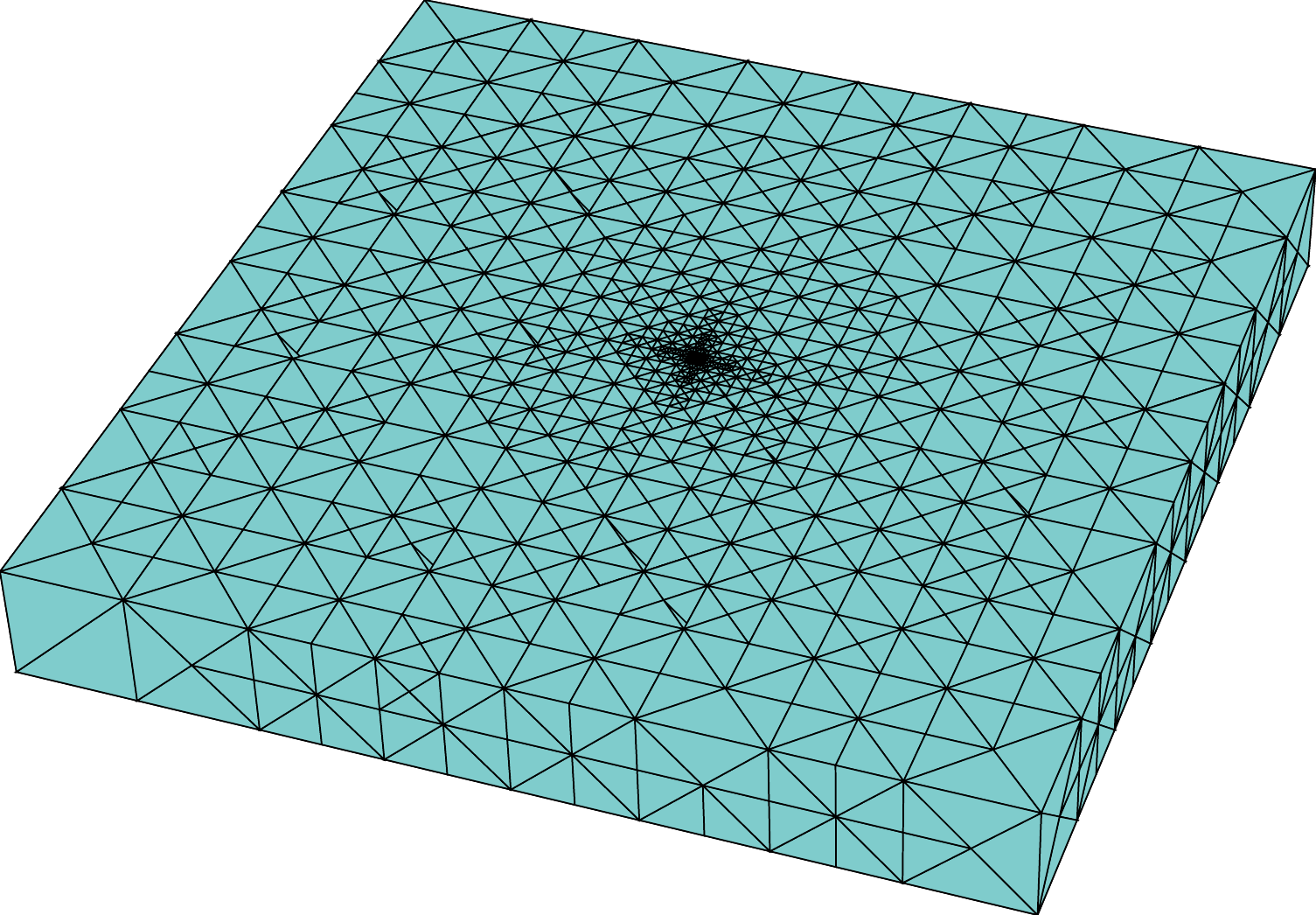}
}
\hskip0.5cm
\subfloat[Refined mesh based on $\eta_{_{\text{New},K}}$ cut on $z\!=\!0$]
{
\includegraphics[width=0.3\textwidth]{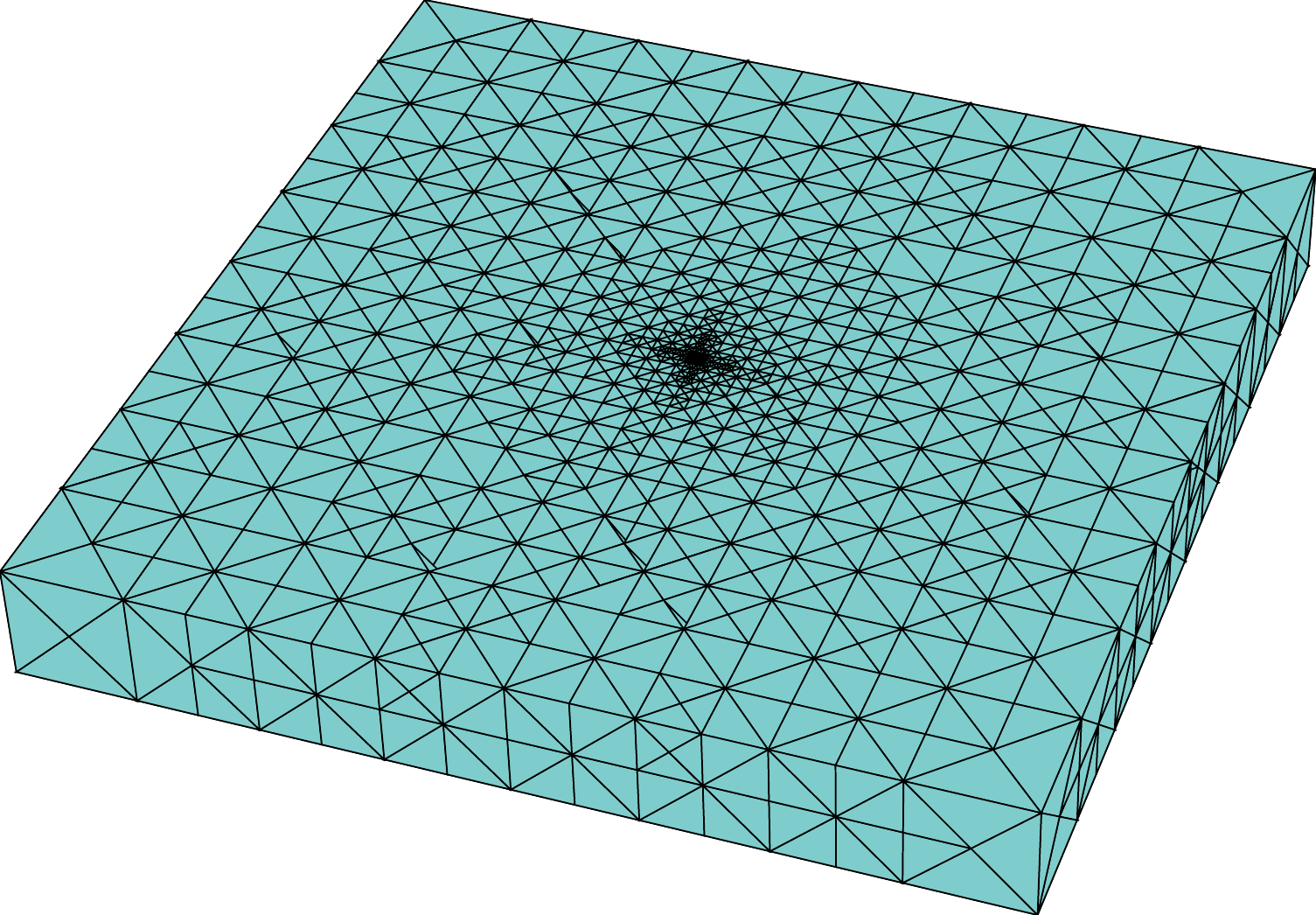}
}
\caption{Mesh results of Example 1, $\divv \vf= 0$} 
\label{fig:ex1-hdiv-mesh}
\end{figure}

It is easy to see that the exact solution of the auxiliary problem in 
(\ref{eq:pb-mf}) for this example is 
$\bsig = 
\mu^{-1}\curlt\vu={\bf 0}$. Hence, the true error for the finite element 
approximation to  (\ref{eq:pb-mf})
is simply the energy norm of the finite element solution $\bsig_{_\cT}$ defined in 
(\ref{eq:pb-mf-weak})
\[
\enorm{ \bsig - \bsig_{_\cT}}_{\mu,\b}^2 
= \enorm{\bsig_{_\cT}}_{\mu,\b}^2
= \norm{\b^{-1/2} \curlt \bsig_{_\cT}}^2 +
\norm{\mu^{1/2} \bsig_{_\cT}}^2.
\]

In the first experiment, we choose the coefficients $\mu= 1$ and $\b = \a$. 
This choice enables that $\divv\vf = 0$, i.e.,
$\vf\in\vHdiv$, and that $\vu$ satisfies the $\b$-weighted normal 
continuity:
\begin{equation}
\label{eq:cont-normal}
\jump{\b\vu\cdot \vn}{F} = 0
\end{equation}
for any surface $F$ in the domain $\Omega$.
This is the prerequisite for establishing efficiency and reliability bounds in 
\cite{Hiptmair2000} and \cite{Cai15} and the base for recovering 
$\btau_h$ in $\BDM_1\subset \vHdiv$ in \cite{Cai15}. The quasi-monotonicity 
assumption is not met in this situation (for the 
analysis of the quasi-monotonicity affects the robustness of the estimator for 
$\vhcrl$ problems, please refer to \cite{Cai15}).

The meshes generated by $\eta_{_{\text{New},K}}$, 
$\eta_{_{\text{Res},K}}$, and $\eta_{_{\text{Rec},K}}$ are almost the same 
(see Figure \ref{fig:ex1-hdiv-mesh}). 
In terms of the convergence, we observe that the 
error estimator $\eta_{_{\text{New}}}$ exhibits  asymptotical
exactness. This is impossible for the error estimators in  
\cite{Hiptmair2000} and \cite{Cai15}  because of the presence of
the element residuals. Table 1 shows that the number of the DoF for the 
$\eta_{_\text{New}}$ is about $30\%$ less than those of the other two estimators 
while achieving a better accuracy.
As the reliability of the estimator does not depend on the quasi-monotonicity of the 
coefficient, the rate of the convergence is not hampered by checkerboard pattern of 
the $\b$.


\begin{table}[h]
\caption{Estimators Comparison, Example 1, $\divv \vf = 0$} 
\centering  
\begin{tabular}{|c|c|c|c|c|c|c|} 
\hline
					& \# Iter & \# DoF & error    & rel-error& $\eta$   & eff-index  
\\ \hline
$\eta_{_\text{Res}}$ & $27$ & $181324$ & $0.0428$ & $0.0494$ & $0.0953$ & $2.226$ 
\\ 
\hline
$\eta_{_\text{Rec}}$ & $27$ & $187287$ & $0.0421$ & $0.0486$ & $0.0428$ & $1.041$  
\\ 
\hline
$\eta_{_\text{New}}$ & $24$ & $127857$ & $0.0411$ & $0.0473$ & $0.0405$ & $0.985$  
\\
\hline
$\wt{\eta}_{_\text{New}}$ & $25$ & $129564$ & $0.0418$ & $0.0482$ & $0.0479$ & 
$1.147$  
\\
\hline
\end{tabular}
\label{table:ex1-1}
\end{table}

In the second experiment, we choose that $\mu = \b = 1$. Due to the fact that 
the normal component of $\vu=\nabla \psi$ is discontinuous across the interfaces, 
the exact solution $\vu$ 
does not satisfy the usual $\b$-weighted normal continuity \eqref{eq:cont-normal}, 
i.e.,
\[
\jump{\b\vu\cdot \vn}{F} \neq 0.
\]
This leads to a right-hand side $\vf = \b\vu = \nabla \psi$ that is not in 
$\vHdiv$ in the primary problem. Even though the $\vhdiv$-continuity is 
required for establishing the reliability and efficiency of the existing 
residual-based 
and recovery-based estimators, the old residual-based and recovery-based estimators 
may 
still be used if $\divv \vf\at{K}\in \vL^2(K)$ for all $K\in\cT$.
Therefore, we implement all three estimators in this experiment as well.

\begin{table}[h!]
\caption{Estimators Comparison, Example 1, $\vf\notin \vHdiv$} 
\centering  
\begin{tabular}{|c|c|c|c|c|c|c|} 
\hline
 					 & \# Iter &\# DoF & error	  & rel-error& $\eta$   & eff-index 
\\
\hline
$\eta_{_\text{Res}}$ & $31$ & $247003$ & $0.0337$ & $0.0507$ & $0.115$  & $3.420$   
\\
\hline
$\eta_{_\text{Rec}}$ & $24$ & $218497$ & $0.0355$ & $0.0534$ & $0.0559$ & $1.577$   
\\
\hline
$\eta_{_\text{New}}$ & $22$ & $99215$  & $0.0342$ & $0.0514$ & $0.0329$ & $0.964$  
\\
\hline
$\wt{\eta}_{_\text{New}}$ & $24$ & $103419$  & $0.0338$ & $0.0508$ & $0.0434$ & 
$1.283$  
\\
\hline
\end{tabular}
\label{table:ex1}
\end{table}

\begin{figure}[h!]
\centering
\subfloat[Refined Mesh based on $\eta_{_{\text{Res},K}}$ cut on $z=0$]
{
\includegraphics[width=0.25\textwidth]{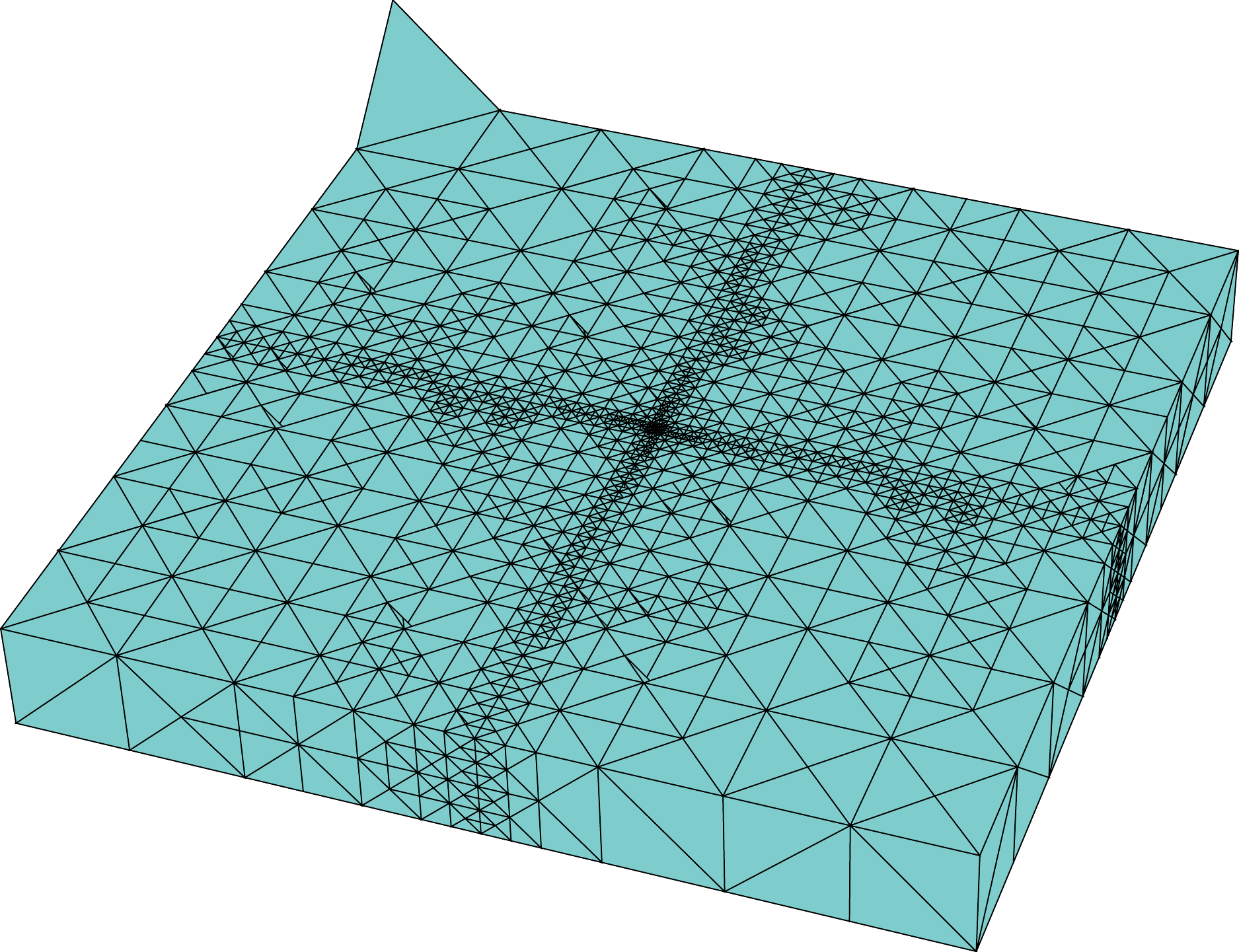}
}
\hskip0.5cm
\subfloat[Refined Mesh based on $\eta_{_{\text{Rec},K}}$ cut on $z=0$]
{
\includegraphics[width=0.25\textwidth]{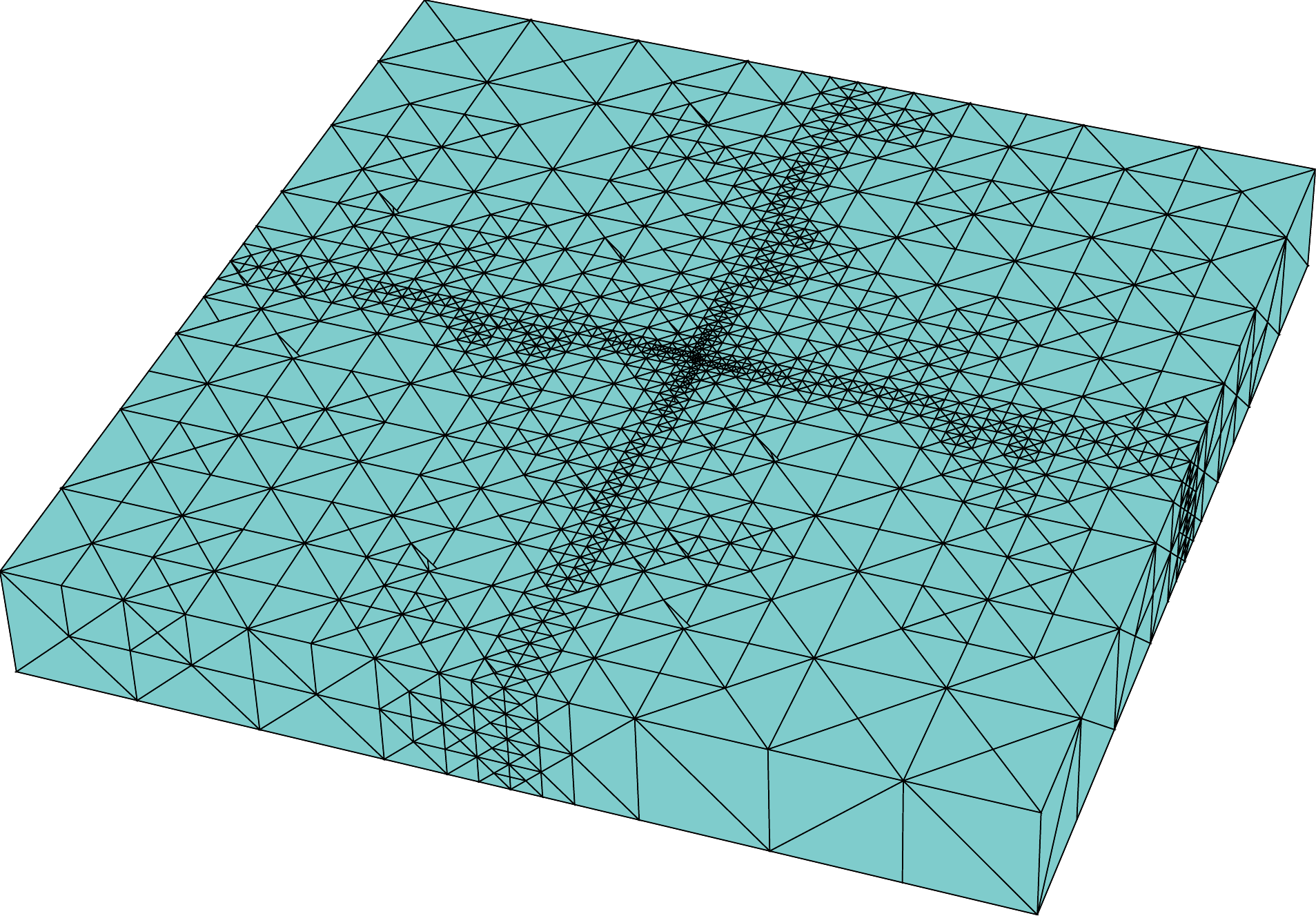}
}
\quad
\subfloat[Refined Mesh based on $\eta_{_{\text{New},K}}$ cut on $z=0$]
{
\includegraphics[scale=0.2]{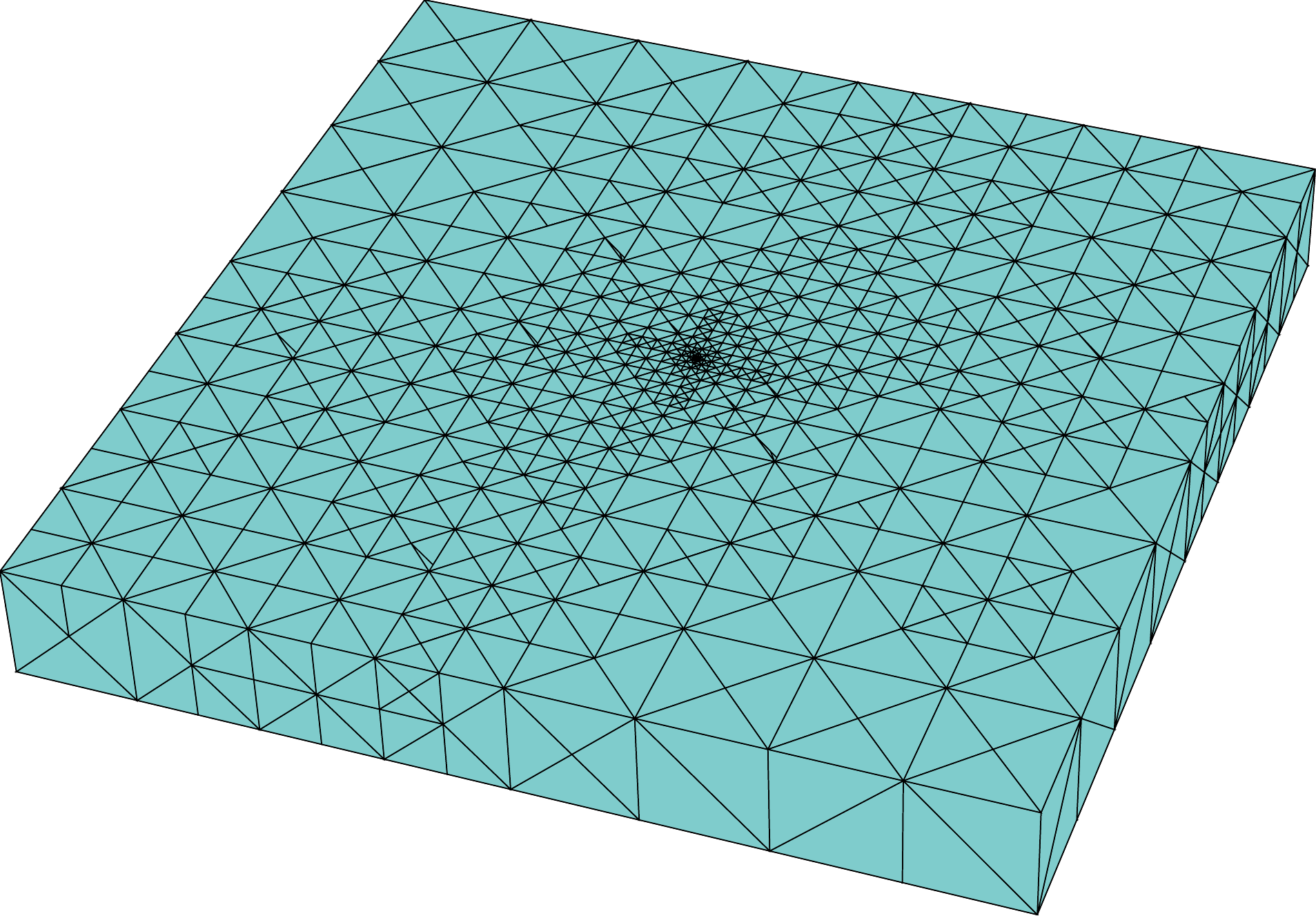}
}
\caption{Mesh Result of Example 1, $\vf\notin \vHdiv$}
\label{fig:ex1-mesh}
\end{figure}

\begin{figure}[h!]
\centering
\subfloat[Convergence of $\eta_{_{\text{Res}}}$]
{
\includegraphics[width=0.28\textwidth]{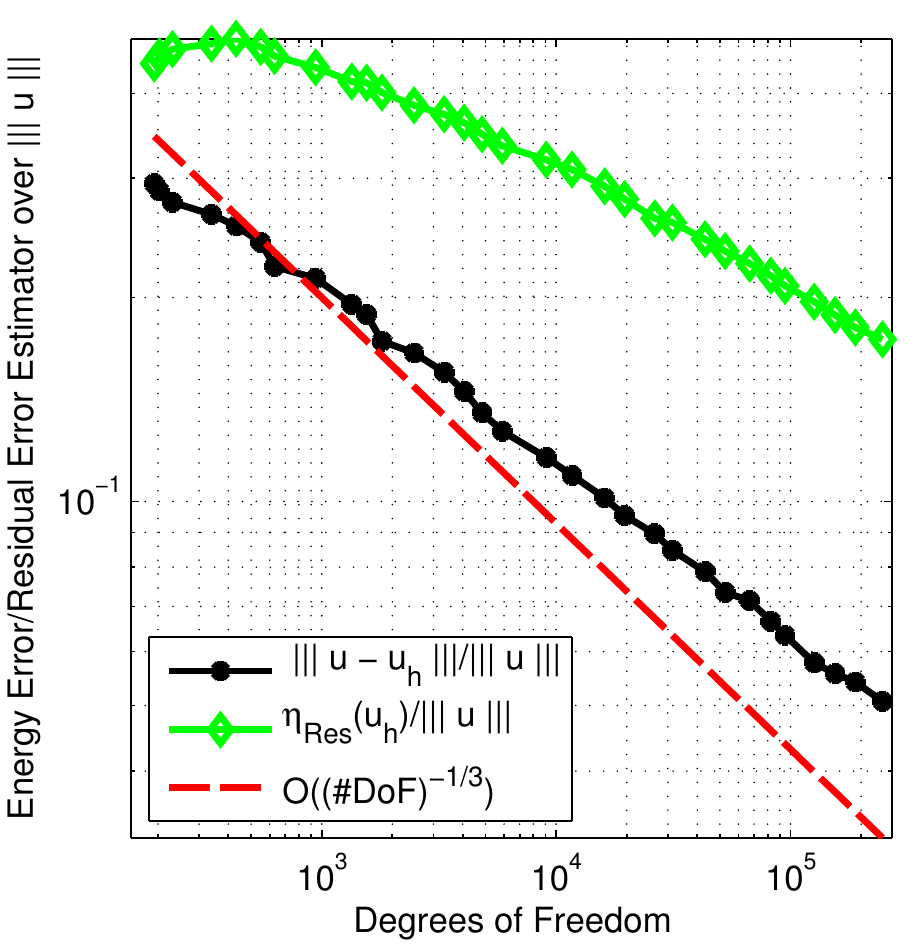}
}
\quad
\subfloat[Convergence of $\eta_{_{\text{Rec}}}$]
{
\includegraphics[width=0.28\textwidth]{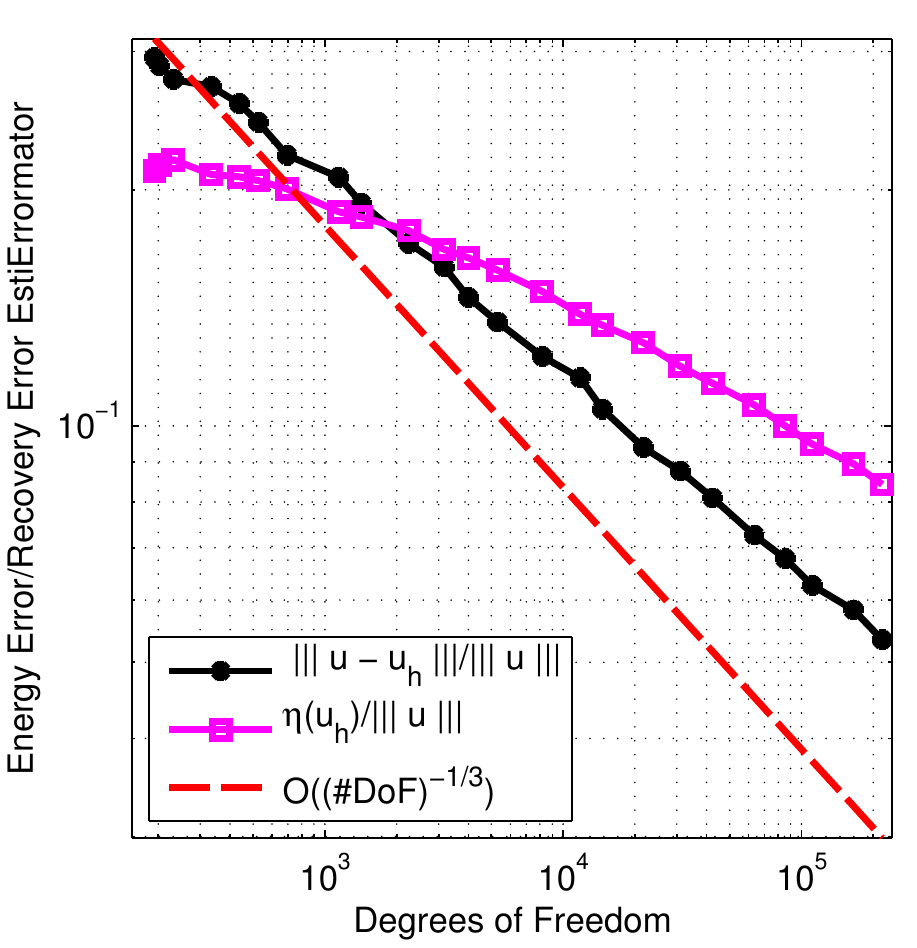}
}
\quad
\subfloat[Convergence of $\eta_{_{\text{New}}}$]
{
\includegraphics[width=0.28\textwidth]{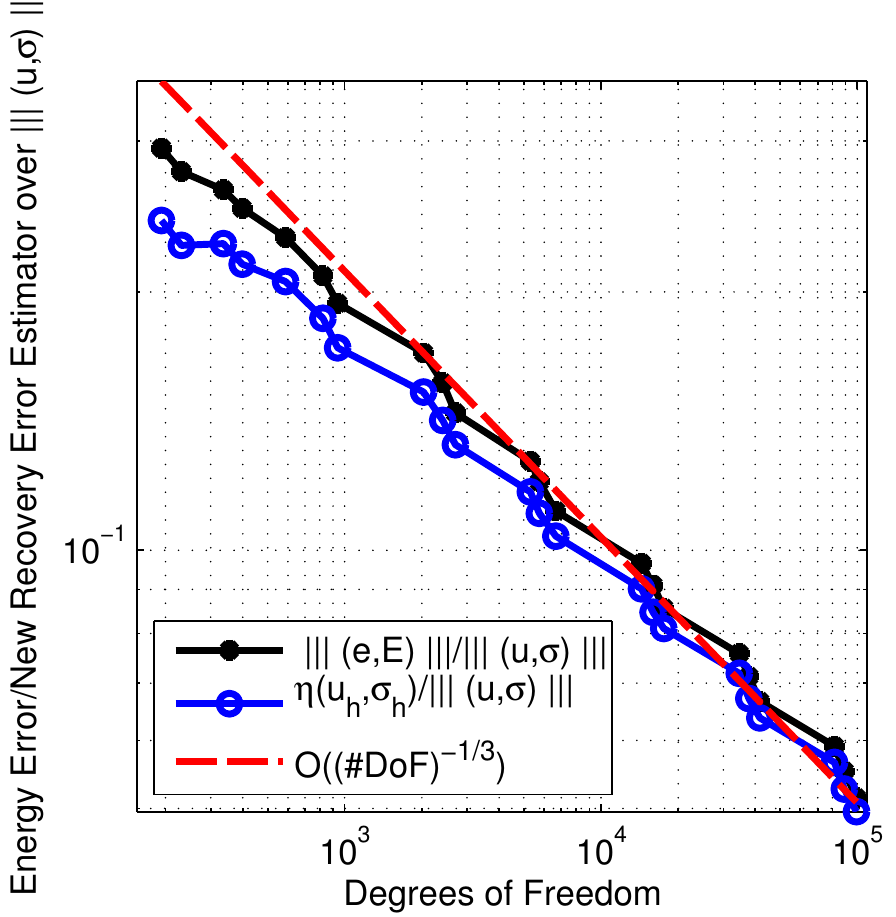}
}
\caption{Convergence Results of Example 1, $\vf\notin \vHdiv$}
\label{fig:ex1-conv}
\end{figure}

For the new estimator $\eta_{_\text{New}}$, it is shown in Figure~4 that the rate of 
convergence is optimal and that the relative true error and the relative estimator 
is approximately equal. 

Table~2 indicates that the number of the degrees of freedom using the 
$\eta_{_\text{New}}$
is less than half of those using the other two estimators. This is confirmed by the 
meshes depicted in Figure~3
where both the $\eta_{_\text{Res}}$ and $\eta_{_\text{Rec}}$ over-refine meshes 
along 
the interfaces, where there are no errors. Such inefficiency of the estimators 
$\eta_{_\text{Res}}$ and $\eta_{_\text{Rec}}$ is also shown in Figure~4 through the 
non-optimal rate of convergence. Moreover, Figure~4 shows 
that both the $\eta_{_\text{Res}}$ and $\eta_{_\text{Rec}}$ are not reliable because 
the slopes of the relative error and the relative estimator are different. The main 
reason for this failure is due to the normal 
jump term $h_F^{1/2}\norm{\jump{\b\vu_{_\cT}\cdot \vn}{F}}_{F}$ along the 
interfaces, 
which is bigger than the true error.

\hspace{0.1in}

\textbf{Example 2}: 
In this example, the performance of the estimators for solenoidal vector field is 
investigated. Like the first example, the coefficients distribution across 
the computational domain is in a checkerboard pattern, not satisfying the 
quasi-monotonicity either.
The computational domain $\overline{\Om} = [-1,1]^3= \overline{\Om}_1\cap 
\overline{\Om}_0$, and $\mu$ is given by:
\[
\mu= 
\begin{cases} 
a & \text{ in } \Om_1,
\\
1 & \text{ in } \Om_0.
\end{cases}
\]
where $\Om_1 = \{(x,y,z)\in \RR^3: xyz>0 \} \cap \Om$ and $\Om_0 = 
\{(x,y,z)\in \RR^3: xyz\leq 0 \} \cap \Om$.
The true solution is given by
\[
\vu = \mu\big(\sin(\pi yz), \sin(\pi xz), \sin(\pi x y)\big).
\]

\begin{table}[h]
\caption{Estimators Comparison, Example 2, $\divv \vf=0$} 
\centering  
\begin{tabular}{|c|c|c|c|c|c|c|} 
\hline
& \# Iter& \# DoF  & error & rel-error & $\eta$  & eff-index 
\\
\hline
$\eta_{_\text{Res}}$  & $34$ & $118740$ & $0.368$ & $0.0699$ & $0.793$ & $2.157$  
\\
\hline
$\eta_{_\text{Rec}}$  & $22$ & $120550$ & $0.365$ & $0.0677$ & $0.644$ & $1.762$   
\\
\hline
$\eta_{_\text{New}}$  & $25$ & $50080$  & $1.500$ & $0.0681$ & $1.490$ & $0.993$  
\\
\hline
$\wt{\eta}_{_\text{New}}$  & $26$ & $51745$  & $1.526$ & $0.0693$ & $1.763$ & 
$1.155$  
\\
\hline
\end{tabular}
\label{table:ex3-hdiv}
\end{table}

In the first experiment, the $\b$ is given as follows:
\[
\b= 
\begin{cases} 
a^{-1} & \text{ in } \Om_1,
\\
1 & \text{ in } \Om_0,
\end{cases}
\]
where $a= 10^{-3}$. This choice makes $\nabla\cdot \vf=0$, and the true 
solution $\vu$ satisfies both the 
tangential continuity and the normal continuity 
\eqref{eq:cont-normal}. Similarly to the first example, the meshes refined using 
three error estimators exhibit no significant difference. Yet, the new estimator 
again shows the asymptotically exactness behavior as Example 1 (see Figure 
\ref{fig:ex3-conv}), and requires much less degrees of freedom to 
achieve the same level of relative error. 
For the results please refer to Table \ref{table:ex3-hdiv}.

\begin{table}[h]
\caption{Estimators Comparison, Example 2, $\vf\notin \vHdiv$} 
\centering  
\begin{tabular}{|c|c|c|c|c|c|c|} 
\hline
					& \# Iter& \# DoF  & error & rel-error & $\eta$  & eff-index  
\\
\hline
$\eta_{_\text{Res}}$ & $31$ & $153352$ & $4.495$ & $0.0699$ & $18.165$ & $4.041$  
\\
\hline
$\eta_{_\text{Rec}}$ & $29$ & $159194$ & $4.538$ & $0.0664$ & $9.857$ & $2.172$   
\\
\hline
$\eta_{_\text{New}}$ & $24$ & $49894$  & $5.263$ & $0.0684$ & $5.254$ & $0.998$   
\\
\hline
$\wt{\eta}_{_\text{New}}$ & $25$ & $57338$  & $4.945$ & $0.0642$ & $6.087$ & 
$1.231$   
\\
\hline
\end{tabular}
\label{table:ex3-1}
\end{table}

\begin{figure}[h!]
\centering
\subfloat[Convergence of $\eta_{_{\text{Res}}}$]
{
\includegraphics[width=0.28\textwidth]{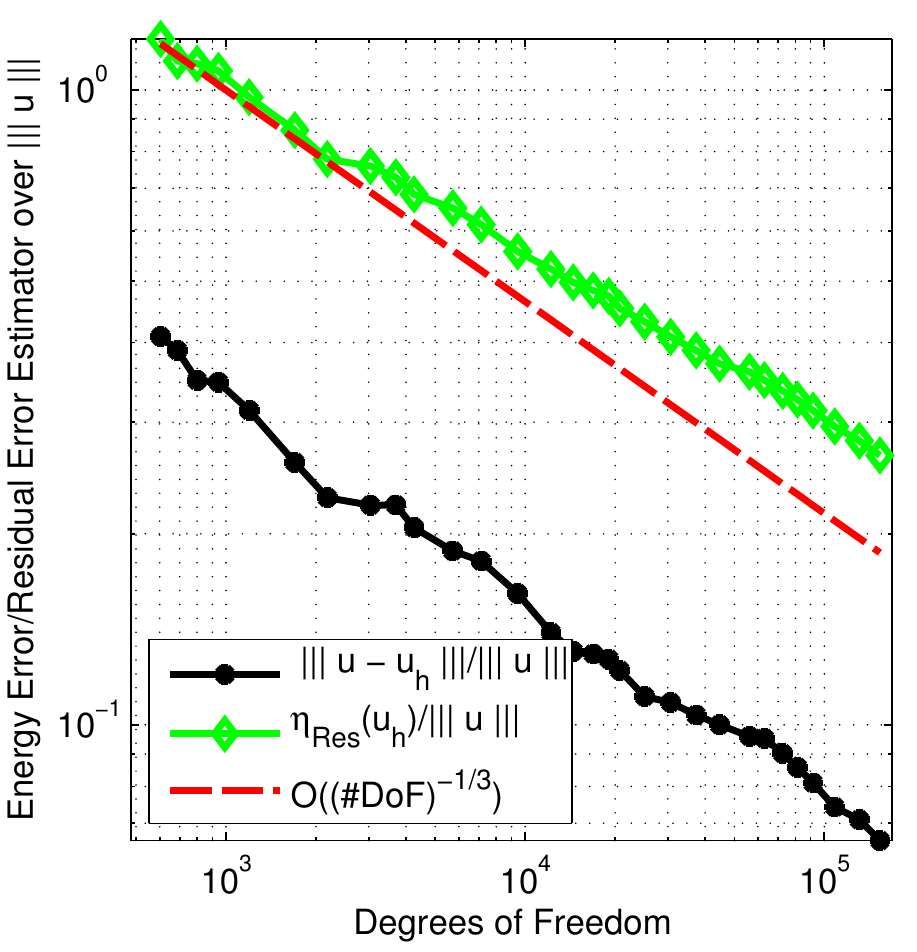}
}
\quad
\subfloat[Convergence of $\eta_{_{\text{Rec}}}$]
{
\includegraphics[width=0.28\textwidth]{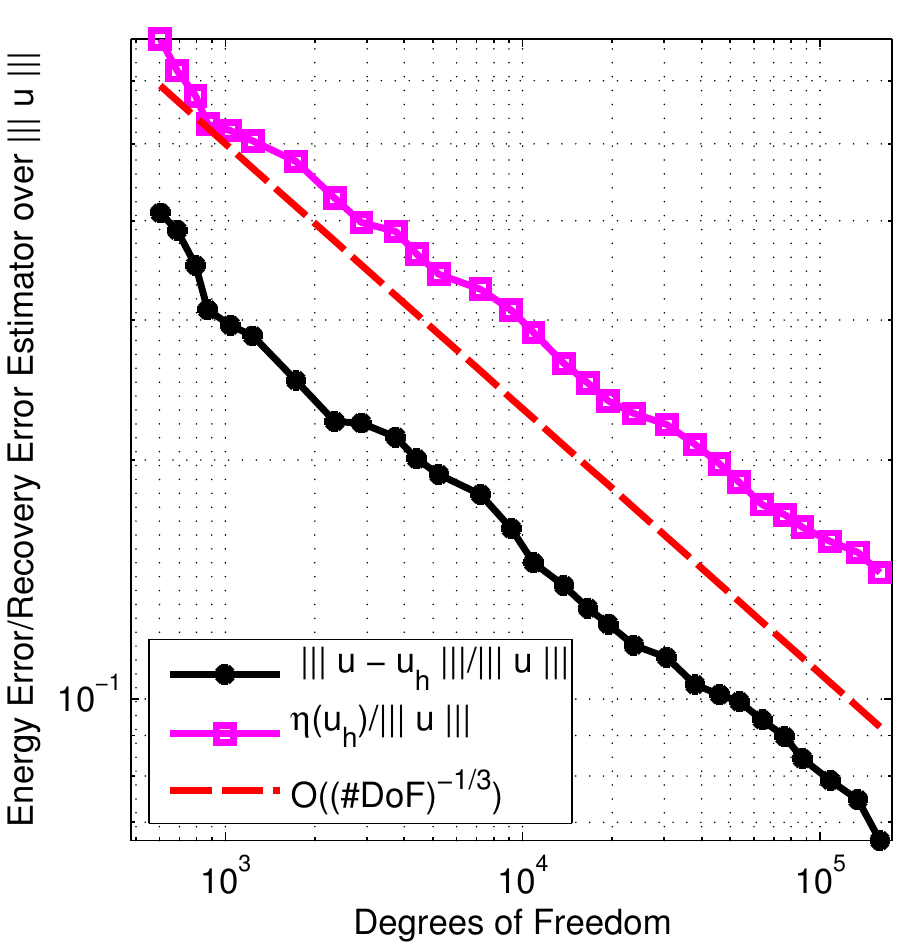}
}
\quad
\subfloat[Convergence of $\eta_{_{\text{New}}}$]
{
\includegraphics[width=0.28\textwidth]{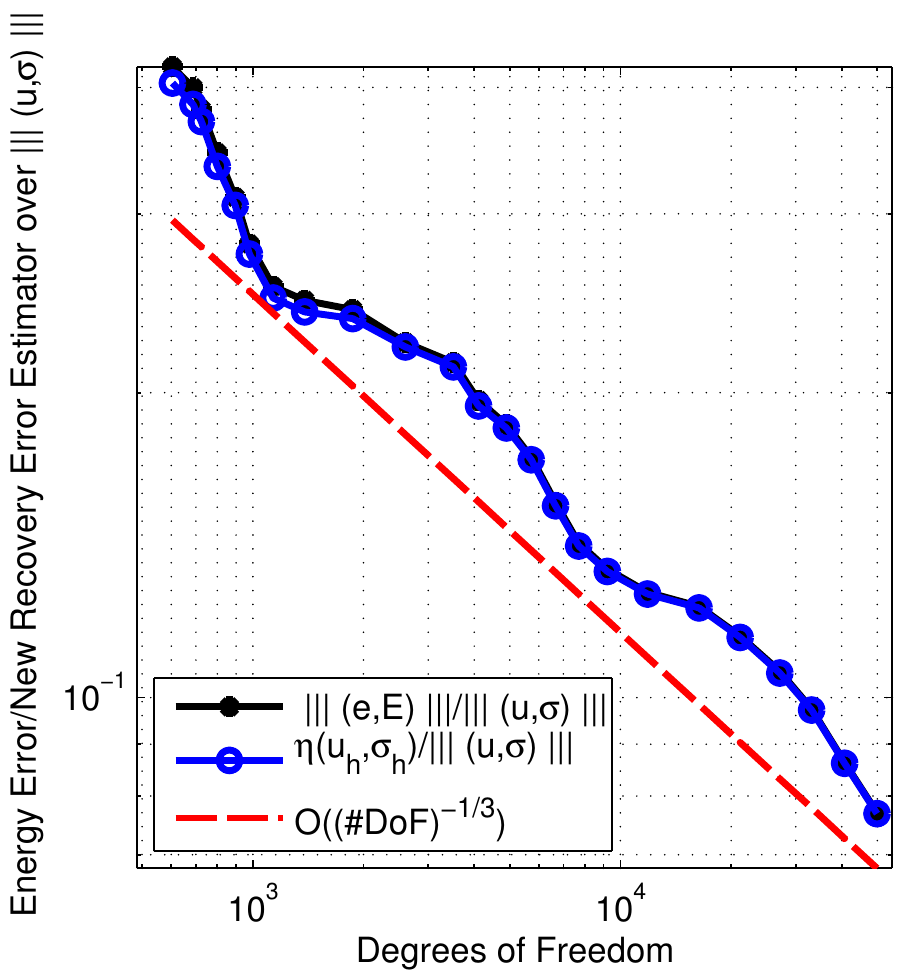}
}
\caption{Convergence Results of Example 2,  $\vf\notin \vHdiv$}
\label{fig:ex3-conv}
\end{figure}

In the second experiment, the $\b$ is chosen to be:
\[
\b= 
\begin{cases} 
1 & \text{ in } \Om_1,
\\
a^{-1} & \text{ in } \Om_0.
\end{cases}
\]
We test the case where $a=10^{-3}$. Similar to Example 1, the necessary tangential 
jump conditions across the interfaces for the primary problem 
are satisfied. Yet the choice of $\b$ implies that the right hand side $\vf 
\notin\vHdiv$. 
Using the residual-based or recovery-based estimator will again lead to unnecessary
over-refinement along the interfaces (see Figure \ref{fig:ex3-mesh}), and the 
order of convergence is sub-optimal than the optimal order for linear 
elements (See Table \ref{table:ex3-1} and Figure \ref{fig:ex3-conv}). 

\begin{figure}[h]
\centering
\subfloat[Adaptively refined mesh based on $\eta_{_{\text{Res},K}}$]
{
\includegraphics[width=0.25\textwidth]{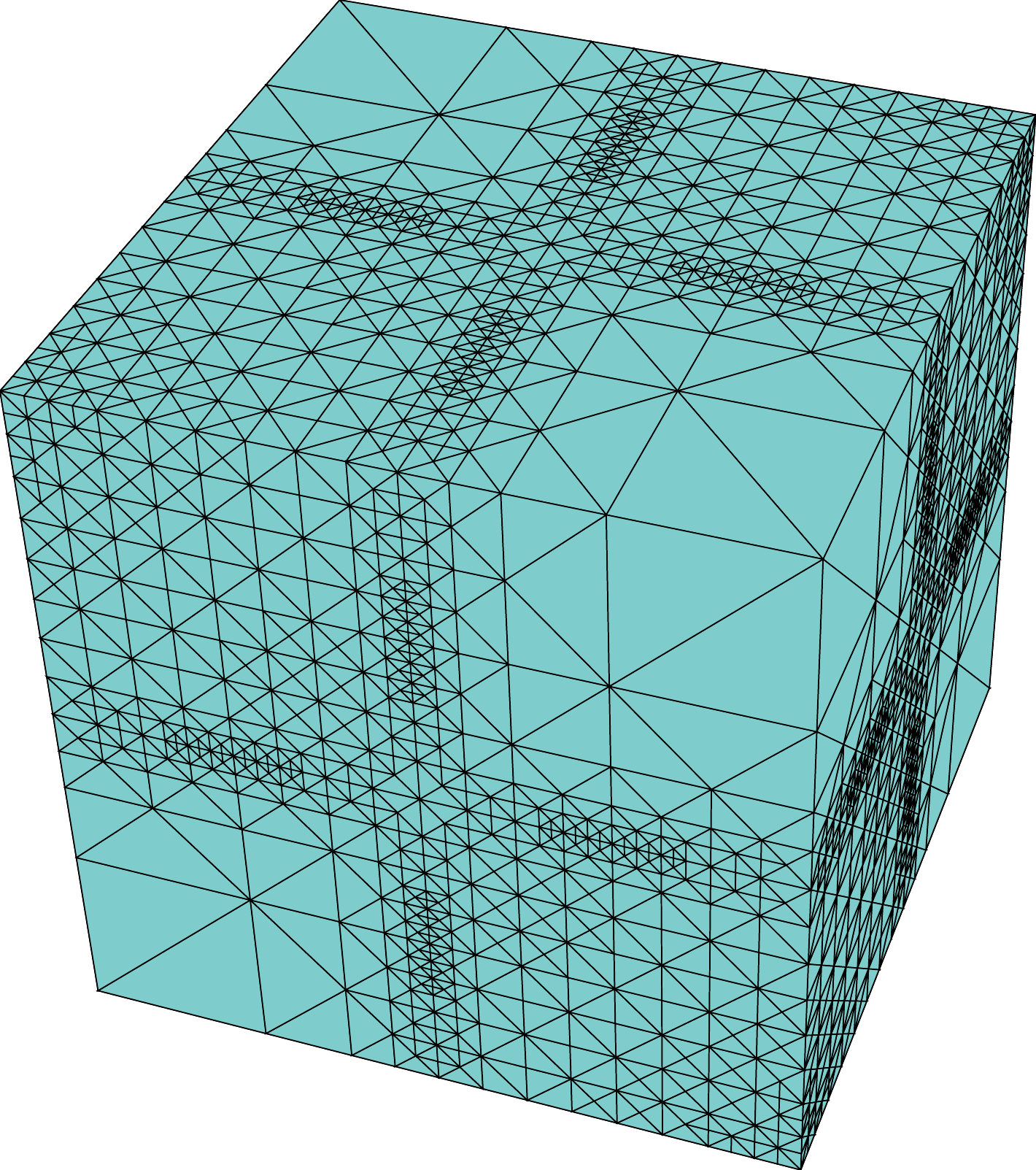}
}
\hskip 0.5cm
\subfloat[Adaptively refined mesh based on $\eta_{_{\text{Rec},K}}$]
{
\includegraphics[width=0.25\textwidth]{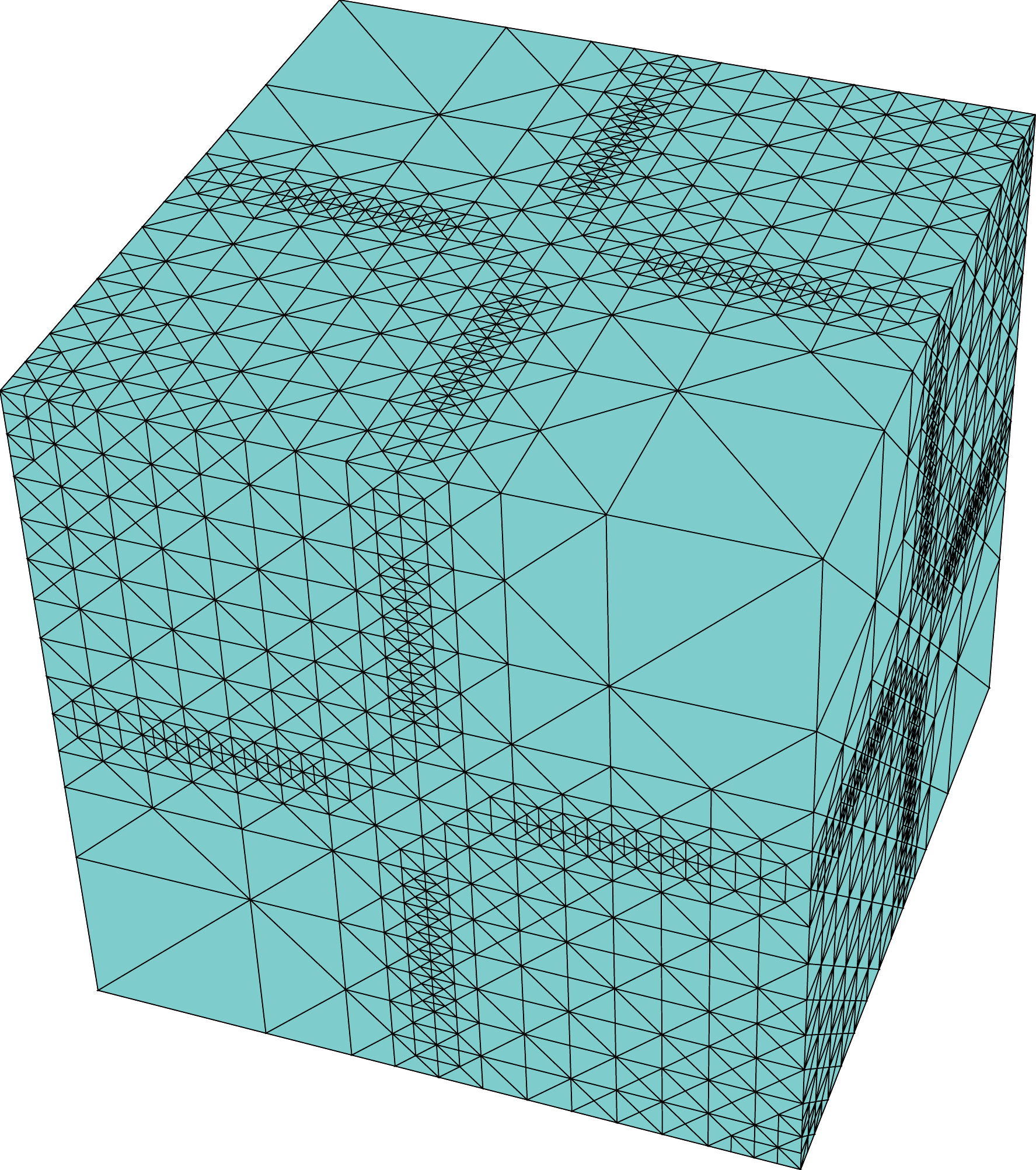}
}
\hskip 0.5cm
\subfloat[Adaptively refined mesh based on $\eta_{_{\text{New},K}}$]
{
\includegraphics[width=0.25\textwidth]{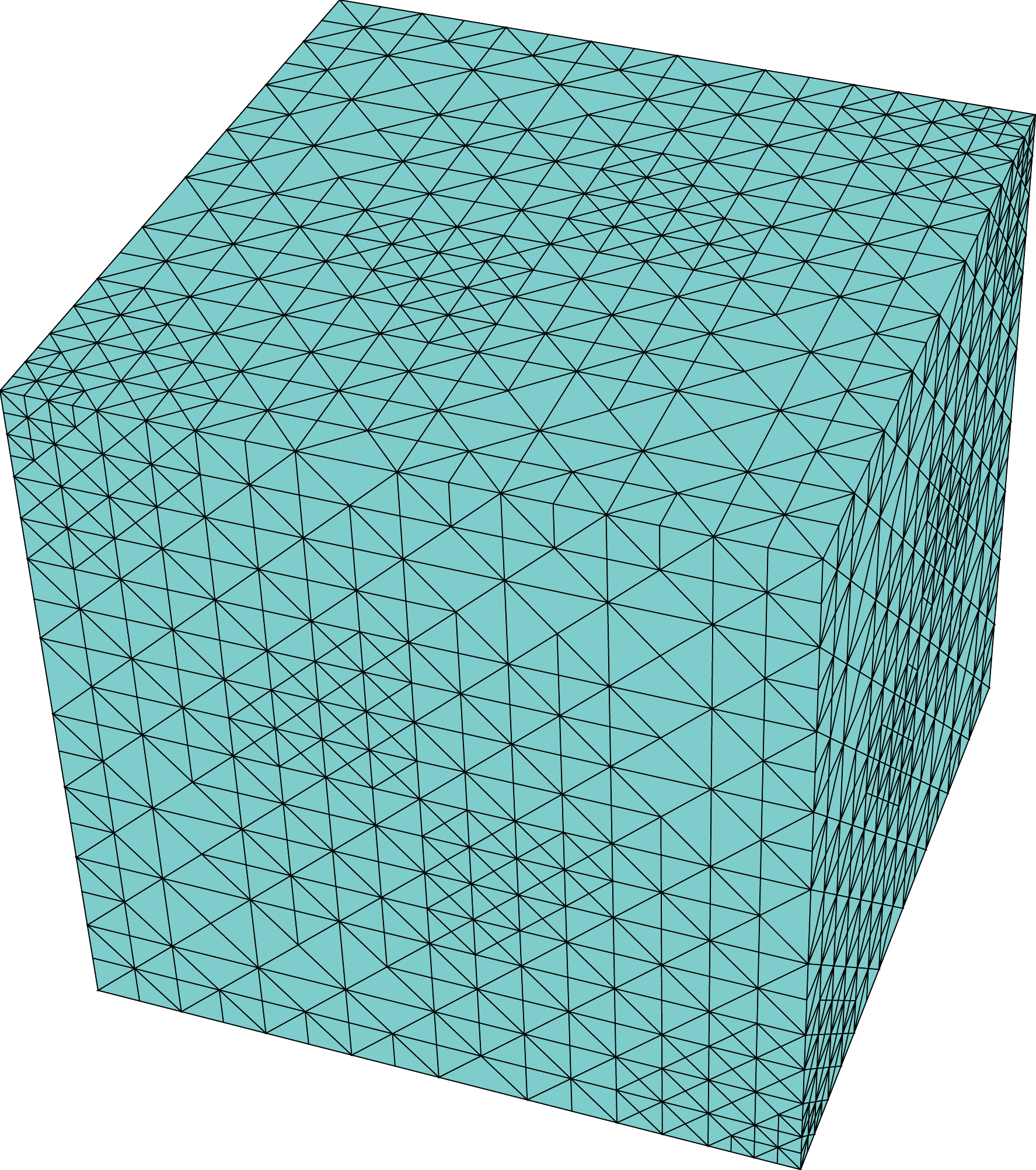}
}
\caption{Mesh Result of Example 2, $\vf\notin \vHdiv$}
\label{fig:ex3-mesh}
\end{figure}

The new estimator in this paper shows convergence in the optimal order no matter 
how we set up the jump of the coefficients. The conclusion of comparison with 
the other two estimators remains almost the same with Example 1. In this 
example, the differences are more drastic: $1/3$ 
the degrees of freedom for the new estimator to get roughly the same level of
approximation with the other two. 


\appendix
\section{Appendix}
\label{appendix}
In this appendix, an a priori estimate for the mixed boundary value 
problem with weights is studied following the arguments and notations mainly from 
\cite{Alonso-Valli,Buffa-Ciarlet}. In our study, it is found that, due to the 
duality pairing on the Neumann boundary and the nature of the trace space of 
$\vhcrl$, a higher regularity is needed for the Neumann boundary data $\vg_{_N}$ 
than those for elliptic mixed boundary value problem. First we define 
the tangential trace operator and tangential component projection operator, and 
their range acting on the $\vHcrlb{B}$. Secondly 
we construct a weighted extension of the Dirichlet boundary data to the interior of 
the domain. Lastly the a priori estimate for the solution of 
problem~\eqref{eq:pb-ef-weak} is established after a trace inequality is set up for 
the piecewise smooth vector field.

\subsection{The tangential trace and tangential component space}

On either Dirichlet or Neumann part of the boundary, the tangential trace operator 
$\g_{\top,B}$ and the tangential component operator $\pi_{\top, B}$ are defined as 
follows:

\begin{equation}
\label{eq:def-trace}
 \g_{\top,B}: \vv\mapsto 
  \left\{\begin{array}{ll}
  \vv\cross\vn  & \mbox{on }\, \G_B,\\[2mm] 
 \bm{0}  & \mbox{on }\,  \p\Om\backslash \ol{\G}_B,
  \end{array}\right.
 \mbox{and}\,\,\,
   \pi_{\top, B}: \vv\mapsto 
   \left\{\begin{array}{ll}
   \vn\cross(\vv\cross\vn) & \mbox{on }\,\G_B, \\[2mm]
   \bm{0} & \mbox{on }\,\p\Om\backslash \ol{\G}_B,
    \end{array}\right.
\end{equation}
respectively, where $\G_B$ is either $\G_{D}$ or $\G_N$.

Define the following spaces as the trace spaces of $\vHcrl$: 
\begin{equation}
\label{eq:space-tr}
\vXbt{\G_B} := \g_{\top, B} \vHcrl\quad
\text{and}\quad  \vXbp{\G_B} := \pi_{\top, B} \vHcrl.
\end{equation}
For the $\vH^1$-regular vector fields, define the trace spaces $\vHhbt{\G_B}$ and 
$\vHhbp{\G_B}$ as:
\begin{equation}
\label{eq:space-tr-hh}
\vHhbt{\G_B} = \pi_{\top,B} \vHo \quad
\text{and}\quad \vHhbp{\G_B} = \g_{\top,B} \vHo.
\end{equation}
It is proved in \cite{Buffa-Ciarlet} that the tangential trace space and the 
tangential component space can be characterized by
\begin{equation}
\label{eq:space-tr-equiv}
\vXbt{\G_B} \subset \vHmhbt{\G_B}
\quad\text{and}\quad 
\vXbp{\G_B} \subset \vHmhbp{\G_B}.
\end{equation}
The $-1/2$ supscripted spaces $\vHmhbt{\G_B}$ and $\vHmhbp{\G_B}$ are defined as the 
dual spaces of $\vHhbt{\G_B}$ and $\vHhbp{\G_B}$.

Now we move on to define the weighted divergence integrable space
\begin{equation}
\begin{aligned}
&\vHdivw{\a} := \{\vv\in \vLt: \divv (\a \vv) \in \Lt \text{ in } \Om \},
\\
\text{and }&\vHdivwz{\a} :=  \{\vv\in \vLt: \divv (\a \vv) =0 \text{ in } \Om \},
\end{aligned}
\end{equation}

and the piecewise regular field space $\vX(\Om,\a,\G_{B})$ as follows:
\begin{equation}
\begin{aligned}
&\vX(\Om,\a,\G_{B}) := \vHcrlbz{B} \cap \vHdivw{\a},
\\
\text{and }& \vX_0(\Om,\a,\G_{B}) := \vHcrlbz{B} \cap \vHdivwz{\a}.
\end{aligned}
\end{equation}
The piecewise $\vH^1$ vector field space is defined as:
\begin{equation}
\label{eq:spaces-piecewise}
\vPH{1} = \{\vv \in \vLt: \vv\at{\Om_j} \in \vH^1(\Om_j), \, j = 1,\ldots,m \}.
\end{equation}

\begin{assumption}[Boundary Requirement]
\label{assumption:bd}
Let the Dirichlet or Neumann boundary $\G_B$ ($B = D$ or $N$) be decomposed into 
simply-connected components: $\ol{\G_B} = \cup_{i} \ol{\G_{B,i}}$. For any 
$\G_{B,i}$, 
there exists a single $j\in \{1,\dots,m\}$, such that 
$\G_{B,i}\subset \p \Om_j\cap \p\Om$.
\end{assumption}


\begin{remark}
Assumption \ref{assumption:bd} is to say, each connected component on 
the Dirichlet or Neumann boundary only serves as the boundary of exactly one 
subdomain. Assumption \ref{assumption:bd} is here solely for the a priori error 
estimate. The robustness of the estimator in Section 5 does not rely on this 
assumption if the boundary data are piecewise polynomials.
\end{remark}

Due to Assumption \ref{assumption:bd}, the tangential trace and tangential component 
of a $\vHo$ vector field is the same space as those of a $\vPH{1}$ vector field on 
$\G_D$ or $\G_N$ respectively. 
With slightly abuse of notation, define 
\begin{equation}
\label{eq:space-tr-hhp}
\vHhbt{\G_B} := \pi_{\top,B} \vPH{1}, \quad\text{and}\quad 
\vHhbp{\G_B} := \g_{\top,B} \vPH{1}.
\end{equation}

Now we define the weighted $1/2$-norm for the value of any 
$\vv\in \vPH{1}\cap \vX(\Om,\b,\G_B)$ on boundary as:
\begin{equation}
\norm{\vv}_{1/2,\b,\mu,\G_B} := \inf_{\vv}
\left\{ \norm{\b^{-1/2}\curlt \vv}^2 + \norm{\mu^{-1/2}\divv(\mu\vv)}^2
+\norm{\mu^{1/2}\vv}^2 \right\}^{1/2}.
\end{equation}
Now thanks to the embedding results from 
\cite{Cai15}, $\norm{\cdot}_{1/2,\b,\mu,\G_B}$ is equivalent to the unweighted 
$\norm{\cdot}_{1/2,\G_B}$ which can be defined as:
\[
\norm{\vv}_{1/2,\G_B} := \inf_{\vv} \left\{ \sum_{j=1}^m \norm{\vv}_{1,\Om_j}^2 + 
\norm{\vv}^2 \right\}^{1/2}.
\]

Naturally, the wegithed $-1/2$-norm of any distribution on the boundary can be 
defined as
\begin{equation}
\label{eq:norm-mhwt}
\norm{\vg}_{-1/2,\mu,\b,\G_{B}} := \sup_{\vv\in \vPH{1}\cap\vX(\Om,\b,\G_{B}) }
\frac{\dualp{\vg}{\vv}_{\G_{B}}}{\norm{\vv}_{1/2,\b,\mu,\G_{B}}}.
\end{equation}

\subsection{Extension of The Dirichlet Boundary Data}
After the preparation, we are ready to construct the extension operator for 
any $\vg_D\in \vXbt{\G_{D}} :=\g_{\top, D} \vHcrl$.

\begin{lemma}
\label{lemma:extension}
For any $\vg_{_D} \in  \vXbt{\G_{D}} $, there exists an extension 
$\vu_{_D}\in \vHcrlb{D}$ such 
that $\g_{\top,D} \vu_{_D} = \vg_{_D}$, and the following estimate holds
\begin{equation}
\enorm{\vu_{_D}}_{\mu,\b} \leq \norm{\vg_{_D}}_{-1/2,\mu,\b,\G_{D}}.
\end{equation}
Moreover, for any $\vv\in \vHcrlbz{D}$, there holds
\[A_{\mu,\beta}(\vu_{_D},\vv) = 0.\]

\end{lemma}
\begin{proof}
The fact that $\vg_{_D} \in  \vXbt{\G_{D}} \subset \vHmhbt{\G_D}$ implies the 
following problem is well-posed:
\begin{equation}
\label{eq:pb-ext}
\left\{
\begin{aligned}
& \text{Find } \vw\in \vPH{1}\cap \vX_0(\Om,\mu,\G_N) \text{ such that:}
\\[2mm]
& \fA_{\b,\mu}(\vw,\vv) = \dualp{\vg_D}{\vv}, \quad \forall\, 
\vv \in \vPH{1}\cap \vX_0(\Om,\a,\G_N),
\end{aligned}
\right.
\end{equation}
where the bilinear form $\fA_{\b,\mu}(\cdot,\cdot)$ is given as
\[\fA_{\b,\mu}(\vw,\vv) := \binprod{\b^{-1}\curlt\vw}{\curlt\vv}
+ \binprod{\mu^{-1}\divv(\mu\vw)}{\divv(\mu\vv)}+ \binprod{\mu\vw}{\vv}.
\]
On this weighted divergence free subspace $\vPH{1}\cap \vX_0(\Om,\mu,\G_N)$:
\[
\fA_{\b,\mu}(\vw,\vv) = \binprod{\b^{-1}\curlt\vw}{\curlt\vv} 
+ \binprod{\mu\vw}{\vv}.
\]
With slightly abuse of notation, the zero extension of $\vg_{_D}$ to the 
Neumann boundary is denoted as $\vg_{_D}$ itself. 
Now for the trial function space and the test function 
space in problem \eqref{eq:pb-ext} are the same, letting $\vw = \vv$ leads to
\[
\enorm{\vw}_{\b,\mu}^2 = \dualp{\vg_{_D}}{\vw}.
\]
Together with their tangential traces vanish on the Neumann boundary, this implies
\[
\enorm{\vw}_{\b,\mu} = \frac{\dualp{\vg_{_D}}{\vw}}{\enorm{\vw}_{\b,\mu}}
\leq \sup_{\vv}\frac{\dualp{\vg_{_D}}{\vv}}{\enorm{\vv}_{\b,\mu}}
\leq \sup_{\vv}\frac{\dualp{\vg_{_D}}{\vv}_{\G_D}}{\norm{\vv}_{1/2,\b,\mu,\G_{D}}}
= \norm{\vg_{_D}}_{-1/2,\mu,\b,\G_{D}}.
\]
The extension is now letting $\vu_{_D} = \b^{-1}\curlt\vw$. To prove the 
estimate, we first notice that the problem \eqref{eq:pb-ext} is a consistent 
variational formulation for the following PDE:
\begin{equation}
\label{eq:pb-ext-strong}
\left\{
\begin{aligned}
\curlt (\b^{-1}  \curlt \vw) + \mu\, \vw &= \bm{0},  
&\text{ in }\, \Om,\;\;
\\
\divv(\mu\vw) &=0,  &\; \text{ in }\, \Om,\;\;
\\
(\mu^{-1} \curlt \vw) \cross \vn & = \vg_{_D}, 
&\; \text{ on }\, \G_{D}.
\end{aligned} 
\right.
\end{equation}
Therefore, the energy norm of $\vu_{_D}$ is
\[
\begin{aligned}
\enorm{\vu_{_D}}_{\mu,\b}^2 &= 
\norm{\mu^{-1/2}\curlt\vu_{_D}}^2 + \norm{\b^{1/2}\vu_{_D}}^2
\\
& = \norm{\mu^{1/2}\vw}^2 + \norm{\b^{-1/2}\curlt\vw}^2
\\
& = \enorm{\vw}_{\b,\mu}^2 \leq \norm{\vg_{_D}}_{-1/2,\mu,\b,\G_{D}}.
\end{aligned}
\]
For the second equality in the Lemma, it is straightforward to verify that for any 
$\vv\in \vHcrlbz{D}$, with $\vw$ is from the above construction, the following 
identity 
holds
\[
\begin{aligned}
A_{\mu,\beta}(\vu_{_D},\vv) &= 
\binprod{\mu^{-1}\curlt\vu_{_D}}{\curlt\vv} + \binprod{\b\vu_{_D}}{\vv}
\\
&= \binprod{-\vw}{\curlt\vv} + \binprod{\curlt\vw}{\vv} 
\\
&= \dualp{\vn\cross\vw}{\vv} = 0.
\end{aligned}
\]
The last equality follows from the fact that $\vw\cross\vn = \bm{0}$ on 
$\G_{N}$ and $\vv\cross\vn = \bm{0}$ on $\G_{D}$.
\end{proof}

\subsection{A Trace inequality}
In this section we want to establish a trace inequality for the tangential 
component space of $\vHcrlb{N}$. For any $\vv\in \vHcrlbz{D}$, consider the 
tangential component space $\vXbp{\G_N}$ defined in \eqref{eq:space-tr} that 
contains 
all the tangential components of $\vv\in \vHcrlbz{D}$ on the Neumann boundary and 
zero 
on the Dirichlet boundary.

\begin{lemma}[Trace inequality for the tangential component]
\label{lemma:trace}
For $\vv\in \vHcrlbz{D}$, the tangential component of $\vv$ on 
$\G_{N}$ is $\pi_{\top,N} \vv \in \vXbp{\G_N}$ and satisfies the following 
estimate:
\begin{equation}
\label{eq:tr}
\norm{\pi_{\top,N} \vv}_{-1/2,\mu,\b,\G_N} \leq
\, \enorm{\vv}_{\mu,\b}.
\end{equation}
\end{lemma}
\begin{proof}
First we notice that $\vXbp{\G_N} \subset \vHmhbp{\G_N}$ which is the dual space of 
$\g_{\top,B} \vPH{1}$. For any $\bxi\in \vHhbp{\G_N} $, there exists $\wt{\bxi} \in 
\vPH{1}$ such that $\wt{\bxi}\cross \vn\at{\G_{N}} =\bxi$ and
\[
\left\{\norm{\b^{-1/2}\curlt \wt{\bxi}}^2 + 
\norm{\mu^{-1/2}\divv(\mu\wt{\bxi})}^2
+\norm{\mu^{1/2}\wt{\bxi}}^2 \right\}^{1/2} \leq \norm{\bxi}_{1/2,\b,\mu,\G_N}.
\]
By the integration by parts formula from \cite{Buffa-Ciarlet} and 
Cauchy-Schwarz inequality, we have
\[
\begin{aligned}
\dualp{\bxi}{\vv}_{\G_{N}} &= \dualp{\bxi}{\pi_{\top,N} \vv}_{\p\Om}
\\
&= \binprod{\curlt \vv}{\wt{\bxi}} - \binprod{\curlt \wt{\bxi}}{\vv}
\\
&= \norm{\mu^{-1/2}\curlt \vv} \norm{\mu^{1/2}\wt{\bxi}}
- \norm{\b^{-1/2}\curlt \wt{\bxi}} \norm{\b^{1/2}\vv}
\\
&\leq \enorm{\vv}_{\mu,\b}\,\enorm{\wt{\bxi}}_{\b,\mu}
\leq \enorm{\vv}_{\mu,\b}\,\norm{\bxi}_{1/2,\b,\mu,\G_N}.
\end{aligned}
\]
Hence by definition \eqref{eq:norm-mhwt} the Lemma follows.
\end{proof}

\subsection{An A Priori Estimate for the \texorpdfstring{$\vhcrl$}{H(curl)} 
Mixed Boundary Value Problem}

\begin{theorem}
Assume that $\vf\in \vL^2(\Omega)$, $\vg_{_D}\in \vXbt{\G_{D}}$, and 
$\vg_{_N}\in \vHhbp{\G_{N}}$. Then the weak formulation of \eqref{eq:pb-ef} 
has a unique solution $\vu \in \vH_D(\curl;\Omega)$ satisfying the following a 
priori estimate
\begin{equation}
\enorm{\vu}_{\mu,\beta} 
\leq \|{\b^{-1/2}\vf}\| 
+\norm{\vg_{_D}}_{-1/2,\mu,\b,\G_{D}} 
+\norm{\vg_{_N}}_{1/2,\b,\mu,\G_N}.
\end{equation}
\end{theorem}

\begin{proof}
Let $\vu_{_D}\in \vHcrlb{D}$ be the extension of the $\vg_{_D}$ to the domain 
$\Omega$ 
from Lemma \ref{lemma:extension} such that
\begin{equation}\label{extension}
\begin{gathered}
\vu_{_D}\cross\vn\at{\G_{D}}= \vg_{_D}, 
\quad \vu_{_D}\cross\vn\at{\G_{N}} = \bm{0},
\\
A_{\mu,\beta}(\vu_{_D},\vv) = 0, \quad\mbox{and}\quad
\enorm{\vu_{_D}}_{\mu,\b} \leq \norm{\vg_{_D}}_{-1/2,\mu,\b,\G_{D}}.
\end{gathered}
\end{equation}

Now let $\vu=\vu_0+\vu_{_D}$, 
then $\vu_0\in \vHcrlbz{D}$ satisfies
\begin{equation}
\label{eq:pb-ef-weak-0}
A_{\mu,\beta}(\vu_0,\vv)  = f_{_N}(\vv),
 \quad \forall\;  \vv \in \vHcrlbz{D}.
\end{equation}
The triangle inequality and \eqref{extension} give
\[
\enorm{\vu}_{\mu,\beta} 
\leq \enorm{\vu_0}_{\mu,\beta} + \enorm{\vu_{_D}}_{\mu,\beta}
\leq \enorm{\vu_0}_{\mu,\beta} + \norm{\vg_{_D}}_{-1/2,\mu,\b,\G_{D}}.
\]

Now, to show the validity of the theorem, 
it suffices to prove that problem \eqref{eq:pb-ef-weak-0} has a unique 
solution 
$\vu_0\in \vHcrlbz{D}$ satisfying the following a priori estimate
\begin{equation}\label{ap-estimate-ef-1}
\enorm{\vu_0}_{\mu,\beta} \leq 
\norm{\b^{-1/2}\vf}+ \norm{\vg_{_N}}_{1/2,\b,\mu,\G_N} .
\end{equation}
To this end, for any $\vv\in \vHcrlbz{D}$, we have from trace Lemma 
\ref{lemma:trace} 
\[
\norm{\pi_{\top,N}\vv}_{-1/2,\mu,\b,\G_N} \leq \enorm{\vv}_{\mu,\b},
\]
which, together with the Cauchy-Schwarz inequality, implies 
\begin{eqnarray*}
|f_{_N}(\vv)|
&\leq & \norm{\b^{-1/2}\vf}\,\norm{\b^{1/2}\vv} + 
\norm{\vg_{_N}}_{1/2,\b,\mu,\G_N}  \norm{\pi_{\top,N}\vv}_{-1/2,\mu,\b,\G_N}
\\[2mm]
&\leq & \left(\norm{\b^{-1/2}\vf}+ \norm{\vg_{_N}}_{1/2,\b,\mu,\G_N}\right)
\enorm{\vv}_{\mu,\beta} .
\end{eqnarray*}
By the Lax-Milgram lemma, \eqref{eq:pb-ef-weak-0} has a unique solution 
$\vu_0\in \vHcrlbz{D}$. Taking $\vv=\vu_0$ in \eqref{eq:pb-ef-weak-0}, we have
\[
\enorm{\vu_0}_{\mu,\beta}^2
= f_{_N}(\vu_0)\leq 
\left(\norm{\b^{-1/2}\vf} 
+ \norm{\vg_{_N}}_{1/2,\b,\mu,\G_N}\right)
\enorm{\vu_0}_{\mu,\beta}  .
\]
Dividing $\enorm{\vu_0}_{\mu,\beta}$ on the both sides of the above 
inequality yields 
\eqref{ap-estimate-ef-1}.
This completes the proof of the theorem.
\end{proof}



\begin{thebibliography}{}

\bibitem{Alonso-Valli}
{\sc A. Alonso and A. Valli},{{ \em Some remarks on the characterization of the 
space of tangential traces of H (rot; $\Omega$) and the construction of an extension 
operator}}, Manuscripta Mathematica, 89-1 (1996), 159--178.

\bibitem{Buffa-Ciarlet}
{\sc A. Buffa and P. Ciarlet, Jr}, {\em On traces for functional spaces related to 
{M}axwell's equations Part {I}: An integration by parts formula in {L}ipschitz 
polyhedra}, Math. Method. Appl. Sci., 24-1 (2001), 9--30.

\bibitem{Hiptmair2000}
{\sc R. Beck, R. Hiptmair, R. W. Hoppe, and B. Wohlmuth}, {\em Residual 
based a posteriori error estimators for eddy current
computation}, Math. Model. Numer. Anal., 34 (2000), 159--182.

\bibitem{Braess06}
{\sc D. Braess and J. Sch\"{o}berl}, {\em Equilibrated residual error 
estimator for edge elements}, Math. Comp., 77 (2008), 
651--672.

\bibitem{Cai12}
{\sc Z.~Cai and S.~Zhang}, {\em Robust equilibrated residual error estimator
  for diffusion problems: Conforming elements}, SIAM J. Numer. Anal., 50
  (2012), pp.~151--170.

\bibitem{Cai15}
{\sc Z. Cai and S. Cao}, {\em A recovery-based a posteriori error estimator 
for $\vhcrl$ interface problems}, Comput. Methods Appl. Mech. Engrg., 
296 (2015), 169-195.


\bibitem{Cai09}
{\sc Z. Cai and S. Zhang}, {\em Recovery-based error estimators for interface 
problems: conforming linear elements}, SIAM J. Numer. Anal., 47-3 (2009), 
2132--2156.

\bibitem{Cai10}
{\sc Z. Cai and S. Zhang}, {\em Recovery-based error estimators for interface 
problems: 
Mixed and nonconforming finite elements}, SIAM J. Numer. Anal., 48 (2010), 
30--52.


\bibitem{Chen.L2008c}
{\sc L. Chen}, {\em {$i$FEM}: an innovative finite element methods package in 
MATLAB}, Technical Report, University of California at Irvine, (2009).


\bibitem{Zou-1}
{\sc J. Chen, Y. Xu, and J. Zou}, {\em An adaptive edge element method and its 
convergence for a saddle-point problem from magnetostatics}, 
Numer. Methods PDEs, 28 (2012), 1643--1666.

\bibitem{Zou-2}
{\sc J. Chen, Y. Xu, and J. Zou}, {\em Convergence analysis of an adaptive edge 
element method for Maxwell's equations}, 
Appl. Numer. Math., 59 (2009), 2950--2969.

\bibitem{Chen10}
{\sc L. Zhong, S. Shu, L. Chen, and J. Xu}, {\em Convergence of adaptive 
edge finite element methods for ${H}(\mathbf{curl})$-elliptic problems}, 
Numerical Linear Algebra with Applications, 17 (2010), 415--432.

\bibitem{Cockburn05}
{\sc B. Cockburn and J. Gopalakrishnan}, {\em Incompressible finite elements via 
hybridization. Part II: The Stokes system in three space dimensions}, 
SIAM J. Numer. Anal., 43-4 (2005), 1651--1672.

\bibitem{Demkowicz2009}
{\sc L. Demkowicz, J. Gopalakrishnan, and J. Sch{\"o}berl}, {\em Polynomial 
extension operators. Part II}, SIAM J. Numer. Anal., 47-5 (2009), 3293--3324.

\bibitem{Ekeland-Temam} {\sc I. Ekeland and R. Temam}, {\em Convex Analysis and 
Variational Problems}, North-Holland, Amsterdam,1976.

\bibitem{Ern15}
  {A. Ern and J.-L. Guermond},
  {\em Mollification in strongly Lipschitz domains with application to continuous 
  and discrete De Rham complex},
  arXiv:1509.01325 [math.NA].
  
\bibitem{Falgout02}
{\sc R. Falgout and U. Yang}, {\em hypre: A library of high performance 
preconditioners}, 
Computational Science-ICCS 2002, (2002), 632--641.

\bibitem{Harutyunyan08}
{\sc F. Izs\'{a}k, D. Harutyunyan, and J. J. W. van der Vegt}, {\em Implicit 
a posteriori error estimates for the {M}axwell equations}, Math. Comp., 
77 (2008), 1355--1386.

\bibitem{Hiptmair98}
{\sc R. Hiptmair}, {\em Multigrid method for Maxwell's equations}, SIAM J. Numer. 
Anal., 36-1 (1998), 204--225.

\bibitem{Hiptmair07}
{\sc R. Hiptmair and J. Xu}, {\em Nodal auxiliary space preconditioning in H(curl) 
and H(div) spaces}, SIAM J. Numer. Anal., 45-6 (2007), 
2483--2509.

\bibitem{Knyazev07}
{\sc A. Knyazev, M. Argentati, I. Lashuk, E. Ovtchinnikov}, {\em Block locally 
optimal preconditioned eigenvalue xolvers (BLOPEX) in HYPRE and PETSc}, SIAM Journal 
on Scientific Computing, 29-5 (2007), 2224--2239.

\bibitem{Kolev06}
{\sc T. Kolev and P. Vassilevski}, {\em Some experience with a $H^1$-based auxiliary 
space AMG for $H(curl)$ problems}, 
Lawrence Livermore Nat. Lab., Livermore, CA, Rep. UCRL-TR-221841, 2006.


\bibitem{mfem-library}
{\em MFEM: Modular finite element methods}, mfem.org.

\bibitem{Monk} {\sc P. Monk}, {\em Finite Element Methods for Maxwell's 
Equations}, Oxford University Press,  2003.

\bibitem{Nedelec80}
{\sc J.-C. N\'{e}d\'{e}lec}, {\em Mixed finite elements in $\mathbb{R}^3$}, 
Numer. Math., 35 (1980), 315--341.

\bibitem{Neittaanmaki2010}
{\sc P. Neittaanm{\"a}ki and S. Repin}, {\em Guaranteed error bounds for conforming 
approximations of a Maxwell type problem}, Applied and Numerical Partial 
Differential Equations, (2010), pp.~199--211.

\bibitem{Nicaise03}
{\sc E. Creus\'{e} and S. Nicaise}, {\em A posteriori error estimation for the 
heterogeneous Maxwell equations on isotropic and anisotropic meshes}, Calcolo, 
40-4 (2003), 249--271.

\bibitem{Nicaise05}
{\sc S. Nicaise}, {\em On Zienkiewicz-Zhu error estimators for Maxwell's 
equations}, C. R. Acad. Sci., Paris, S\'{e}r. I, 
340 (2005), 697--702.

\bibitem{Nicaise07}
{\sc S. Cochez-Dhondt and S. Nicaise}, {\em Robust a posteriori error 
estimation for the Maxwell equations}, Comput. Methods Appl. Mech. Engrg., 
196 (2007), 2583--2595.

\bibitem{Oden89}
{\sc J.T. Oden and L. Demkowicz and W. Rachowicz and T.A. Westermann}, {\em Toward a 
universal hp adaptive finite element strategy, 
Part 2. A posteriori error estimation}, Comput. Methods Appl. Mech. Engrg., 
77-1 (1989), 113--180.

\bibitem{Repin2007}
{\sc S. Repin}, {\em Functional a posteriori estimates for Maxwell's equation}, 
Journal of Mathematical Sciences, 142-1 (2007), 1821--1827.

\bibitem{Petzoldt02}
{\sc M. Petzoldt}, {\em A posteriori error estimators for elliptic equations 
with discontinuous coefficients}, Advances in Computational Mathematics, 
16(2002), pp.~47--75.

\bibitem{Prager1947}
{\sc W. Prager and J.L. Synge}, {\em Approximations in elasticity based on the 
concept of function space}, Quart. Appl. Math, 
5-3(1947), pp.~241--269.

\bibitem{Schoberl07}
{\sc J. Sch\"{o}berl}, {\em A posteriori error estimates for 
Maxwell equations}, Math. Comp., 77 (2008), 633--649.

\end{thebibliography}
\end{document}